\documentclass[a4paper]{amsart}
\usepackage[T1]{fontenc}
\usepackage{amsmath,amsthm,amsfonts,url,breakurl}
\usepackage{paralist,mathtools}
\usepackage[usenames,dvipsnames]{color}
\usepackage[pdftex]{hyperref}
\usepackage[sort,numbers]{natbib}

\newcommand{\eps}{\varepsilon}

\numberwithin{equation}{section}
\newtheorem{theo}{Theorem}

\newtheorem{lemma}[theo]{Lemma}
\newtheorem{prop}[theo]{Proposition}
\newtheorem{cor}[theo]{Corollary}
\newtheorem{defi}[theo]{Definition}

\theoremstyle{definition}
\numberwithin{theo}{section}
\newtheorem{rems}[theo]{Remarks}
\newtheorem{rem}[theo]{Remark}

\newcommand{\Jnup}{J^{\uparrow n}}
\newcommand{\Jndown}{J^{\downarrow n}}

\title[Convergence of semigroups and generalised elliptic forms]{Convergence of operator semigroups associated with generalised elliptic forms}
\keywords{Elliptic sesquilinear forms, Semigroup convergence, Gibbs semigroups, Minimax principle, Schatten classes}
\subjclass[2010]{Primary: 47D06; Secondary: 47B10}
\author{Delio Mugnolo}
\address{Delio Mugnolo\\Institut f\"ur Analysis\\Universit\"at Ulm\\89069 Ulm\\Germany}
\email{delio.mugnolo@uni-ulm.de}

\author{Robin Nittka}
\address{Robin Nittka\\Institut f\"ur Angewandte Analysis\\Universit\"at Ulm\\89069 Ulm\\Germany}
\email{robin.nittka@uni-ulm.de}

\DeclareMathOperator{\sgn}{sign}
\DeclareMathOperator{\Real}{Re}
\DeclareMathOperator{\Ima}{Im}
\DeclareMathOperator{\Kern}{ker}
\newcommand{\scalar}[3][]{\left(#2\mid#3\right)_{#1}}

\begin{document}

\begin{abstract}
In a recent article, Arendt and ter Elst have shown that every sectorial form
is in a natural way associated with the generator of an analytic strongly
continuous semigroup, even if the form fails to be closable. As an
intermediate step they have introduced so-called \emph{$j$-elliptic} forms,
which generalises the concept of elliptic forms in the sense of Lions. We push
their analysis forward in that we discuss some perturbation and convergence
results for semigroups associated with $j$-elliptic forms. In particular, we
study convergence with respect to the trace norm or other Schatten norms. We
apply our results to Laplace operators and Dirichlet-to-Neumann-type operators.
\end{abstract}

\maketitle

\section{Introduction}

The use of sesquilinear form in semigroup theory dates back to the works of
Tosio Kato and Jacques-Louis Lions. A generalisation of Kato's and Lions'
approach has been recently proposed by Wolfgang Arendt and Tom ter
Elst~\cite{AreEls09}. Their method permits to treat differential operators on
rough domains, strongly degenerate equations, Dirichlet-to-Neumann operators
and Stokes-type equations with ease, cf.~\cite{AreEls09,AreEls11}. In this
article we consider only what they call the \emph{complete case}, which
corresponds to Lions' forms, not their \emph{incomplete case}, which
corresponds to Kato's approach. These two notions are different descriptions of
the same ideas. In Section~2 we introduce $j$-elliptic forms and recall some
basic facts which we need. We also prove that $j$-ellipticity is preserved
under small perturbations and we also present a generalisation of the Courant's
minimax formula. 

The study of convergence of sequences of $C_0$-semigroups goes back to the
pioneering works on semigroup theory in the 1950s. In particular, convenient
convergence criteria for semigroups associated with closed forms can be found
in Kato's book~\cite{Kat95}. In Section~3 we establish criteria for
$j$-elliptic forms that imply strong convergence of the associated semigroups.
Such convergence results will in turn allow us to deduce convergence in
stronger norms, for example Schatten norms. Our first result in this section is
a Mosco-like convergence criterion for symmetric forms (Theorem~\ref{thm:strongconv}).

The Schatten classes $\mathcal L_p$ have been introduced
in~\cite{SchNeu46} by Robert Schatten and John von Neumann. For $p=1$,
one obtains the well-studied trace class. It became clear soon after the publication of~\cite{SchNeu46}
that trace class operators play an important r\^ole in spectral theory,
perturbation theory and mathematical physics. An interesting account on the history of the
development of the Schatten theory can be found in the introduction
of~\cite{Sim05}. Criteria for convergence of a sequence of operators with
respect to Schatten norms have been investigated for a long time, see
e.g.~\cite{Zag80} and references therein. We translate a result
due to Valentin A.\ Zagrebnov into the framework of $j$-elliptic forms,
which gives a sufficient condition for convergence in Schatten norm. We then
combine this with an interpolation result for Schatten class operators in order to
prove convergence of semigroups as Schatten class operators into spaces of higher
regularity, e.g.\ from $L^2(\Omega)$ into $H^k(\Omega)$ for any $k \in \mathbb{N}$.

Summarizing our main results, on $L^2(X)$, $X$ a finite measure space, the following holds:
\begin{quotation}
Strong convergence implies trace norm (hence uniform) convergence of a family of self-adjoint contraction semigroups, provided that their generators all dominate the generator of an ultra-contrac\-tive semigroup; and this even as operators from $L^2(X)$ into a space of more regular functions.
\end{quotation}
This is made precise in Corollary~\ref{cor:summarize} and the subsequent remark.

In Section~4 we present several applications for our theorems and ideas. More
precisely, we study Schatten norm convergence of semigroups generated by
Laplacians with varying Robin boundary conditions as well as trace norm
convergence of semigroups generated by Dirichlet-to-Neumann-like operators with
varying coefficients. We also compare the spectra of several self-adjoint
operators based on our general version of the minimax formula.

\section{Generalised elliptic forms}\label{j-f:sec}

In this section we study \emph{$j$-elliptic forms}. We start with some basic facts.
For a broader introduction and proofs of the fundamental theorems we refer to~\cite{AreEls09}.

\begin{defi}\label{def:complete}
Let $V$ and $H$ be Hilbert spaces and $j\colon V \to H$ a bounded linear map
with dense range. A sesquilinear form $a\colon V \times V \to \mathbb{C}$ is
called a $j$-elliptic form on $H$ with form domain $V$ if it is continuous as a
function from $V \times V$ to $\mathbb{C}$ and there exist $\omega \in \mathbb{R}$ and $\mu > 0$
such that
\begin{equation}\label{eq:elliptic}
	\Real a(u,u) - \omega\|j(u)\|^2_H \ge \mu \|u\|_V^2\qquad \hbox{ for all }u \in V.
\end{equation}
The unique, densely defined, m-sectorial operator $A$ on $H$ given by
\begin{align*}
	D(A) & \coloneqq \{ x \in H : \exists u \in V, \; j(u) = x,\; \exists f\in H\hbox{ s.t. } \; a(u,v) = \scalar[H]{f}{j(v)} \; \forall v \in V \} \\
	Ax & \coloneqq f
\end{align*}
is called the operator associated with $(a,j)$. We say that $(a,j)$ is associated to $A$
and also that $(a,j)$ is associated with the analytic $C_0$-semigroup $(e^{-tA})_{t \ge 0}$ on $H$.

We say that $a$ is symmetric if $a(u,v) = \overline{a(v,u)}$ for all $u,v\in V$.
In this case the associated operator $A$ and the semigroup $(e^{-tA})_{t \ge 0}$ are self-adjoint.

We say that $a$ is positive if $a(u,u) \ge 0$ for all $u\in V$. By the polarisation identity every
positive (or, more generally, every real-valued) sesquilinear form is symmetric.
\end{defi}

If $j$ is injective, we can regard $V$ as a subspace of $H$, regarding $j$ as the embedding
of $V$ into $H$. In this case the notion of a $j$-elliptic form $a$ introduced in Definition~\ref{def:complete}
coincides with Lions' definition of elliptic forms.
Thus we refer to this situation as \emph{classical}, i.e., we say that $(a,j)$ is a \emph{classical form}
if $j$ is injective.

\begin{rem}\label{jclass}
Let $a$ be a $j$-elliptic form. Then
\begin{equation}\label{eq:Va}
	V(a) \coloneqq \{ u \in V : a(u,v) = 0 \; \forall v \in \Kern j \}
\end{equation}
is a closed subspace of $V$, $j_{|V(a)}$ is injective and $V = V(a) \oplus
\Kern j$. In particular, $j(V(a)) = j(V)$ is a dense subspace of $H$ and $j_{|V(a)}$ is
injective. The classical form $(a_{|V(a)\times V(a)},j_{|V(a)})$ is associated
to the same operator as $(a,j)$. This relation allows us to carry over many
results about classical forms to $j$-elliptic forms, which is the basis of this section.
\end{rem}

\begin{rem}\label{rem:remelliptic}
Remark~\ref{jclass} suggests that $a$ is associated with an m-sectorial operator if merely
\[
	\Real a(u,u) - \omega\|j(u)\|^2_H \ge \mu \|u\|_V^2\qquad \hbox{ for all }u \in V(a).
\]
This is indeed true provided that we require $V = V(a) + \Kern j$ in addition, cf.~\cite[Cor.~2.2]{AreEls09}.
\end{rem}

We want to extend several classical results to $j$-elliptic forms. We
begin with a generation result which is a translation of a
celebrated by Michel Crouzeix on cosine function generators.

\begin{prop}\label{lemma:crouz}
Let $a$ be a $j$-elliptic form and denote by $A$ the associated operator.
Assume that there exists $M \ge 0$ such that
\begin{equation}
\label{crouzj}
|\Ima a(u,u)|\le M\|u\|_V \|j(u)\|_H\qquad \hbox{for all }u\in V(a).
\end{equation}
Then $-A$ generates a cosine
operator function and hence a semigroup with analyticity angle of
$\frac{\pi}{2}$.
\end{prop}

\begin{proof}
For all $x\in D(A)$ with $\|x\|_H=1$ there exists $u\in V(a)$ such that
$j(u) = x$ and
\begin{align*}
|\Ima (Ax|x)_H|^2 &=|\Ima a(u,u)|^2
\le M^2\|u\|_{V}^2 \|j(u)\|_H^2\\
&\le \frac{M^2}{\mu} \left(\Real a(u,u) - \omega \|j(u)\|^2_H \right) \|j(u)\|^2_H\\
&= \frac{M^2}{\mu} \left(\Real (Ax|x)_H - \omega  \right).
\end{align*}
Thus, the numerical range of $A$ is contained in a parabola and therefore
$-A$ generates a cosine operator function by Crouzeix' celebrated result~\cite{Cro04}.
Finally, every generator of a cosine function family generates a holomorphic
semigroup of angle $\frac{\pi}{2}$~\cite[Thm.~3.14.17]{AreBatHie01}.
\end{proof}

The following perturbation results are analogous to two classical
perturbation theorems for operators~\cite{DesSch84,DesSch88},
one relying on interpolation estimates, the other one on compactness.


\begin{prop}\label{lemma:perturb}
Let $a:V\times V\to \mathbb C$ be a $j$-elliptic form
and let $H'$ be a subspace of $H$ containing $j(V)$. Let $H'$ carry
its own norm $\|\cdot\|_{H'}$, for which it is a Banach space and is continuously embedded into $H$.
Assume that there exist $\alpha \in [0,1)$ and $M \ge 0$ such that
$$\|j(u)\|_{H'}\le M \|u\|_V^\alpha \|j(u)\|_H^{1-\alpha}\qquad \text{for all } u\in V.$$
Let $b:V\times V\to \mathbb C$ be a continuous sesquilinear form such that
$$\Real b(u,u)\ge -c\|u\|_V \|j(u)\|_{H'}\qquad\hbox{for all } u\in V$$
for some $c \ge 0$. Then $a+b:V\times V\to \mathbb C$ is $j$-elliptic.
\end{prop}

\begin{proof}
We apply Young's inequality $\alpha \beta \le \eps \alpha^p + c_{\eps,p}
\beta^{p/(p-1)}$, which is valid for every $p \in (1,\infty)$, every
$\alpha,\beta\ge 0$, and every $\eps > 0$ with some constant $c_{\eps,p} \ge 0$.
For $p \coloneqq \frac{2}{1+\alpha}$ we obtain that
\begin{align*}
	\Real b(u,u) & \ge -c \|u\|_V \|j(u)\|_H
		\ge -c M \|u\|_V^{1+\alpha} \|j(u)\|_H^{1-\alpha} \\
		& \ge -c M \eps \|u\|_V^2 - c M c_{\eps,p} \|j(u)\|_H^2
\end{align*}
for all $u \in V$. For $\eps \coloneqq \frac{\mu}{2cM}$ we thus obtain that
\[
	\Real a(u,u) + \Real b(u,u) - (\omega - c M c_{\eps,p}) \|j(u)\|_H^2 \ge \frac{\mu}{2} \|u\|_V^2
\]
for all $u \in V$, which is the claim.
\end{proof}

\begin{rem}
In the classical case Proposition~\ref{lemma:perturb} coincides with~\cite[Lemma~2.1]{Mug08}.
\end{rem}

For the second perturbation theorem we need the following simple lemma.
\begin{lemma}\label{lem:eberlein}
	Let $V$ be a reflexive Banach space, $T\colon V \to H$ an injective bounded linear operator
	into a Banach space $H$ and $S\colon V \to Z$ a compact linear operator into a Banach
	space $Z$. Then for every $\eps > 0$ there exists $c_\eps \ge 0$ such that
	\[
		\|Su\|_Z \le \eps \|u\|_V + c_\eps \|Tu\|_H\qquad \hbox{for all }u \in V.
	\]
\end{lemma}
\begin{proof}
	Assume to the contrary that there exist $\eps_0 > 0$ and a sequence $(u_n)_{n\in\mathbb N} \subset V$ such that
	\[
		\|Su_n\|_Z \ge \eps_0 \|u_n\|_V + n \|Tu_n\|_H\qquad \hbox{for all }n\in\mathbb N.
	\]
	We can assume that $\|Su_n\|_Z = 1$ after rescaling. Passing to a subsequence
	we have $u_n \rightharpoonup u$ in $V$, hence
	$Tu_n \rightharpoonup Tu$ in $H$. Now $\|Tu_n\|_H \le \frac{1}{n}$ implies that $Tu = 0$
	and thus $u = 0$. Hence by compactness $\lim_{n\to \infty}Su_n= Su = 0$ in $Z$, contradicting
	$\|Su_n\|_Z = 1$.
\end{proof}

The conclusion of the following perturbation result should be compared with Remark~\ref{rem:remelliptic}.
\begin{prop}\label{prop:perturb}
	Let $a$ be a $j$-elliptic form on $V$. Let $S$ be a compact operator from
	$V$ into a Banach space $Z$ and let $b_0\colon V \times Z \to \mathbb{C}$
	be a bounded sesquilinear form.	Define $b(u,v) \coloneqq b_0(u,Sv)$ on $V
	\times V$. If $j$ is injective on $V(a+b)$, where $V(a+b)$ is defined
	as in~\eqref{eq:Va}, then there exist $\omega' \in
	\mathbb{R}$ and $\mu' > 0$ such that
	\begin{equation}
	\label{eq:jremark}
		\Real a(u,u) + \Real b(u,u) - \omega' \|j(u)\|_H^2 \ge \mu' \|u\|_V^2 \quad\text{for all } u \in V(a+b).
	\end{equation}
\end{prop}

\begin{proof}
	Regarding $j$ as an injective operator on $V(a+b)$, from Lemma~\ref{lem:eberlein} we obtain that
	\[
		\|Su\|_Z \le \eps \|u\|_V + c_\eps \|j(u)\|_H
	\]
	for all $u \in V(a+b)$. Hence
	\begin{align*}
		|b(u,u)|
			& = |b_0(u,Su)|
			\le c \|u\|_V \|Su\|_Z \\
			& \le \eps c \|u\|_V^2 + c_\eps c \|u\|_V \|j(u)\|_H
			\le \eps c \|u\|_V^2 + \delta c_\eps c \|u\|_V^2 + \frac{c_\eps c}{4\delta} \|j(u)\|_H^2
	\end{align*}
	for $u \in V(a+b)$ by Young's inequality. If we first pick $\eps > 0$ small enough and then $\delta > 0$,
	we easily obtain the claimed estimate from the $j$-ellipticity of $a$.
\end{proof}

Strictly speaking, the preceding result is not quite a perturbation result because
we leave the class of $j$-elliptic forms. It is, however,
quite useful in situations where one cannot expect that a lower order perturbation preserves $j$-ellipticity,
see~\cite[\S 4.4]{AreEls09} for such an example.

We continue our investigation of $j$-elliptic forms with results about domination and convergence.
It is well-known that domination of self-adjoint operators in terms of their resolvents
can be expressed via their quadratic forms. One implication of this characterisation remains
true for symmetric $j$-elliptic forms. The following proposition is a direct consequence of Remark~\ref{jclass}
and~\cite[Thm.~VI.2.21]{Kat95}.

\begin{prop}\label{prop:dom}
	Let $H$ be a Hilbert space, let $a_1$ be a symmetric $j_1$-elliptic form and let $a_2$ be a symmetric $j_2$-elliptic form, where
	$j_1\colon V_1 \to H$ and $j_2\colon V_2 \to H$.
	Let $A_i$ be the self-adjoint operator on $H$ which is associated with $a_i$, $i=1,2$.
	We say that $a_1$ \emph{lies above} $a_2$ (and write $(a_1,j_1) \ge (a_2,j_2)$) if
	\begin{enumerate}[(1)]
	\item $j_1(V_1) \subset j_2(V_2)$ and
	\item $a_1(u_1,u_1) \ge a_2(u_2,u_2)$ whenever $j_1(u_1) = j_2(u_2)$.
	\end{enumerate}
	In this case $(\gamma + A_1)^{-1} \le (\gamma + A_2)^{-1}$ in the sense of positive definite operators
	for all sufficiently large $\gamma \in \mathbb{R}$.
\end{prop}

We also give a result concerning the domination of the spectra in the case where reference spaces $H$ differ.
The following is an easy consequence of the Courant--Fischer theorem for self-adjoint operators (or, rather,
their quadratic forms) and Remark~\ref{jclass}.

\begin{lemma}\label{lem:minimax}
Let $a$ be a symmetric $j$-elliptic form on a Hilbert space $H$ with form
domain $V$ and associated operator $A$. If $j$ is compact, then the
self-adjoint operator $A$ has compact resolvent, and we can order the eigenvalues
of $A$ in increasing order, i.e.,
\[
	\lambda_1(A) \le \lambda_2(A) \le \lambda_3(A) \le \dots \le \lambda_n(A) \to \infty,
\]
taking into account multiplicities. In this case, the eigenvalues are given by the min-max principle
\[
	\lambda_k(A) = \min_{\substack{E \subset V(a) \\ \dim E = k}} \max_{\substack{u \in E \\ u \neq 0}} \frac{a(u,u)}{\|j(u)\|_H^2},
\]
i.e., $E$ runs over the $k$-dimensional subspaces of $V(a)$.
\end{lemma}

The following theorem allows the comparison of
operators on different spaces that have comparable $j$-elliptic forms.
\begin{theo}\label{thm:compspec}
	Let $V_1$, $V_2$, $H_1$ and $H_2$ be Hilbert spaces such that $V_2$ is a closed subspace of $V_1$,
	which is equipped with the norm of $V_1$.
	Let $a_1$ be a symmetric $j_1$-elliptic form, where $j_1 \colon V_1 \to H_1$ is compact,
	and let $a_2$ be a symmetric $j_2$-elliptic form, where $j_2 \colon V_2 \to H_2$ is bounded.
	Assume that $\Kern j_1 \subset V_2$ and that
\begin{equation}\label{eq:courantgeneral}
	\|j_1(u)\|_{H_1} \ge \|j_2(u)\|_{H_2}\qquad \hbox{and}\qquad a_1(u,u) \le a_2(u,u)\qquad \hbox{for all }u \in V_2.
\end{equation}
	Then $j_2$ is compact and $\lambda_k(A_1) \le \lambda_k(A_2)$ for all $k \in \mathbb{N}$, where
$A_1$ and $A_2$ are the operators  associated with $a_1$ and $a_2$ on $H_1$ and $H_2$, respectively.
\end{theo}

\begin{proof}
	Let $(u_n)$ be a bounded sequence in $V_2$.
	Then $(u_n)$ is a bounded sequence in $V_1$. Passing to a subsequence we can assume that
	$(j_1(u_n))$ converges in $H_1$. Since
	\[
		\|j_2(u_n) - j_2(u_m)\|_{H_2} \le \|j_1(u_n) - j_1(u_m)\|_{H_1}
	\]
	by~\eqref{eq:courantgeneral} this implies that $(j_2(u_n))$ is a Cauchy sequence in $H_2$,
	hence convergent. We have proved compactness of $j_2$.

	For the spectral domination it suffices to consider the following three special cases:
	\begin{enumerate}[(i)]
	\item
		$a_2 = a_1|_{V_2 \times V_2}$ and $j_2 = j_1|_{V_2}$; or
	\item
		$V_1 = V_2$ and $j_1 = j_2$; or
	\item
		$V_1 = V_2$ and $a_1 = a_2$.
	\end{enumerate}
	In fact, once we have established the result in these situations, we obtain that
	\[
		\lambda_k(a_1,j_1)
			\le \lambda_k(a_1|_{V_2 \times V_2},j_1|_{V_2})
			\le \lambda_k(a_2, j_1|_{V_2})
			\le \lambda_k(a_2,j_2)
		\quad (k \in \mathbb{N}).
	\]
	Here we have defined $\lambda_k(a,j) \coloneqq \lambda_k(A)$ with $A$ associated to $(a,j)$ to keep the notation simple.
	It should be noted that $a_1|_{V_2 \times V_2}$ is $j_1|_{V_2}$-elliptic since $V_2 \subset V_1$
	and $a_2$ is $j_1|_{V_2}$-elliptic since $a_2(u,u) \ge a_1(u,u)$ on $V_2$.

	So let us prove the theorem in those three cases.
	\begin{enumerate}[(i)]
	\item
		Assume that $a_2 = a_1|_{V_2 \times V_2}$ and $j_2 = j_1|_{V_2}$.
		Since $\Kern j_1 \subset V_2$, this implies that $\Kern j_1 = \Kern j_2$.
		Thus trivially $V(a_2) \subset V(a_1)$, see~\eqref{eq:Va}, implying that every
		subspace of $V(a_2)$ is a subspace of $V(a_1)$. Hence $\lambda_k(A_1) \le \lambda_k(A_2)$
		for all $k \in \mathbb{N}$ by Lemma~\ref{lem:minimax}.
	\item
		Assume that $V_1 = V_2 \eqqcolon V$ and $j_1 = j_2 \eqqcolon j$.
		Let $k \in \mathbb{N}$ be arbitrary and fix a subspace $E_2$ of $V(a_2)$ with $\dim E_2 = k$ such that
		\[
			\lambda_k(A_2) = \max_{\substack{u \in E_2 \\ u \neq 0}} \frac{a_2(u,u)}{\|j(u)\|^2}
		\]
		Then in particular
		\begin{equation}\label{eq:lkA2}
			\lambda_k(A_2) \ge \max_{\substack{u \in E_2 \\ u \neq 0}} \frac{a_1(u,u)}{\|j(u)\|^2}
		\end{equation}
		by~\eqref{eq:courantgeneral}. Define
		\[
			E_1 \coloneqq \{ u \in V(a_1) : j(u) \in j(E_2) \}.
		\]
		Since $j$ is bijective from $V(a_1)$ and $V(a_2)$ to $j(V)$, respectively, see Remark~\ref{jclass},
		we have $\dim E_1 = k$, thus
		\begin{equation}\label{eq:lkA1}
			\lambda_k(A_1) \le \max_{\substack{u \in E_1 \\ u \neq 0}} \frac{a_1(u,u)}{\|j(u)\|^2}
		\end{equation}
		by Lemma~\ref{lem:minimax}.
		In view of~\eqref{eq:lkA2} and~\eqref{eq:lkA1} the theorem is proved once we show that for every $u \in E_1$
		there exists $\tilde{u} \in E_2$ such that $a_1(u,u) \le a_1(\tilde{u},\tilde{u})$ and $j(u) = j(\tilde{u})$.

		Thus fix $u \in E_1 \subset V(a_1)$. By definition of $E_1$ there exists $\tilde{u} \in E_2$
		such that $j(u) = j(\tilde{u})$. By Remark~\ref{jclass} there exist $\tilde{u}_1 \in V(a_1)$ and $\tilde{u}_2 \in \Kern j$
		such that $\tilde{u} = \tilde{u}_1 + \tilde{u}_2$, so in particular $j(u) = j(\tilde{u}) = j(\tilde{u}_1)$.
		Since $j$ is injective on $V(a_1)$, this implies that $u = \tilde{u}_1$. Hence
		\begin{align*}
			a_1(\tilde{u},\tilde{u})
				& = a_1(u + \tilde{u}_2, u + \tilde{u}_2) \\
				& = a_1(u,u) + 2 \Real a_1(u, \tilde{u}_2) + a_1(\tilde{u}_2,\tilde{u}_2)
				\ge a_1(u,u)
		\end{align*}
		since $a_1(u,\tilde{u}_2) = 0$ by definition of $V(a_1)$ and $a_1(\tilde{u}_2,\tilde{u}_2) \ge 0$
		by~\eqref{eq:elliptic}.
	\item
		Assume that $V_1 = V_2$ and $a_1 = a_2 \eqqcolon a$.
		From~\eqref{eq:courantgeneral} we obtain that $\Kern j_1 \subset \Kern j_2$, which implies
		$V(a_2) \subset V(a_1)$. Now we can proceed as in the first case.
	\qedhere
	\end{enumerate}
\end{proof}

For semigroups on $L^2(\Omega)$ associated with classical forms,
ultra-contractivity is well-known to be equivalent to an embedding of the form
domain into $L^q(\Omega)$ for $q > 2$, provided that the semigroup extends to a
contractive semigroup on $L^\infty(\Omega)$.
We translate this result into the language of $j$-elliptic forms, which will be
useful in the subsequent sections when we study Gibbs semigroups.

\begin{prop}\label{prop:ultra}
Let $\Omega$ be a $\sigma$-finite measure space. Let $a$ be a $j$-elliptic form on $H \coloneqq L^2(\Omega)$
with form domain $V$ and associated operator $A$. Assume that there exists $M \ge 0$
such that $\|e^{-tA} f\|_\infty \le M \|f\|_\infty$ for all $f \in L^\infty(\Omega) \cap L^2(\Omega)$
and all $t \in [0,1]$.
Assume moreover that $j(V)\subset L^\frac{2d}{d-2}(\Omega)$ for some $d > 2$.
Then $(e^{ta})_{t\ge 0}$ is ultra-contractive, i.e., $e^{-tA}L^2(\Omega) \subset L^\infty(\Omega)$ and
$$\|e^{-tA}\|_{{\mathcal  L}(L^2 ,L^\infty)} \le ct^{-\frac{d}{4}},\qquad t \in (0,1],$$ for some constant $c>0$. 
\end{prop}
\begin{proof}
By the closed graph theorem $j$ is bounded from $V$ to $L^{\frac{2d}{d-2}}(\Omega)$.
Thus the result follows from Remark~\ref{jclass} and~\cite[Thm.~6.4]{Ouh05}.
\end{proof}

\section{Convergence results}\label{schatten:sec}

Several results in~\cite{AreEls09} are based on a convergence result~\cite[Thm.~3.9]{AreEls09}.
We extend this criterion in the case of symmetric forms.
It is well-known that for symmetric classical forms the convergence in the sense of
Mosco, see~\cite{Mos67}, is equivalent to strong convergence of the resolvents. In fact, this
holds even in the nonlinear case and is typically stated
only in that situation. We show how this criterion translates to $j$-elliptic
forms.

\begin{theo}\label{thm:strongconv}
	Let $(a_n,j_n)_{n\in\mathbb N}$ and $(a,j)$ be positive forms on a Hilbert space $H$
	with form domains $(V_n)_{n\in\mathbb N}$ and $V$, respectively.
	We assume that $a_n$ is $j_n$-elliptic for all $n \in \mathbb{N}$ and $a$ is $j$-elliptic.
	Then the following are equivalent.
	\begin{enumerate}[(a)]
	\item
		The sequence of operators $(-A_n)_{n\in\mathbb N}$ associated with $(a_n,j_n)_{n\in\mathbb N}$ converges to the operator $-A$ associated with $(a,j)$ in the strong resolvent sense.
	\item
		The following conditions are satisfied:
		\begin{enumerate}[(i)]
		\item
			If $u_n \in V_n$, $j_n(u_n) \rightharpoonup x$ for some $x \in H$ and $\liminf_{n \to \infty} a_n(u_n,u_n) < \infty$, then
			there exists $u \in V$ such that $j(u) = x$ and $\liminf_{n \to \infty} a_n(u_n,u_n) \ge a(u,u)$;
		\item
			For all $u \in V$ there exists a sequence $(u_n)_{n\in \mathbb N}$ with $u_n\in V_n$ such that 
			$$\lim_{n\to \infty}j_n(u_n) = j(u)\qquad \hbox{and}\qquad \liminf_{n\to \infty}a_n(u_n,u_n) \le a(u,u).$$
		\end{enumerate}
	\end{enumerate}
	If these equivalent conditions are satisfied, we say that \emph{$(a_n,j_n)_{n\in\mathbb N}$ converges to $(a,j)$ in the sense of Mosco}.
\end{theo}

\begin{proof}
	Define $\phi_n(j_n(u)) \coloneqq a_n(u,u)$ for $u \in V_n(a_n)$, and $\phi_n(x) \coloneqq \infty$
	for $x \in H \setminus j_n(V_n)$. 
	Then $\phi_n\colon H \to (-\infty,\infty]$ is well-defined, convex
	and lower semicontinuous, and $-A_n$ is the subdifferential of $\phi_n$.
	This follows from~\cite[Thm.~2.5]{AreEls09} and the well-known correspondence between the linear and
	the non-linear theory of forms.
	Moreover,
	\begin{equation}\label{eq:minrep}
		\phi_n(x) = \min\{ a(u,u) : u \in V_n, \; j_n(u) = x \}
	\end{equation}
	by Remark~\ref{jclass}.
	A similar statement holds for the functional $\phi$, which we define analogously for $(a,j)$.

	The two conditions in (b) are equivalent to
	\begin{enumerate}[(I)]
	\item
		$x_n \rightharpoonup x$ implies that $\liminf_{n \to \infty} \phi_n(x_n) \ge \phi(x)$;
	\item
		for all $x \in H$ there exists $(x_n) \subset H$ such that 
		$$\lim_{n\to \infty}x_n = x\qquad \hbox{and}\qquad \lim_{n\to \infty}\phi_n(x_n) = \phi(x).$$
	\end{enumerate}
	In fact, assume (i) and (ii). If $\liminf_{n \to \infty} \phi_n(x_n) = \infty$ in (I), then there is nothing
	to show. Otherwise, (I) follows from (i) and~\eqref{eq:minrep}. In (II), if $x \not\in j(V)$, i.e., $\phi(x) = \infty$,
	then by (I) any sequence $(x_n)$ in $H$ such that $\lim_{n\to \infty}x_n = x$ does the job.
	On the other hand, if $x = j(u)$ for some $u \in V(a)$, then (II) follows from (ii) and (I).
	On the contrary, if (I) and (II) are satisfied, then (i) and (ii) follow easily using~\eqref{eq:minrep}.

	We have shown that condition (b) is equivalent to Mosco-convergence of $\phi_n$ to $\phi$,
	which  by~\cite[Prop.~3.19 and Thm.~3.26]{Att84} is equivalent
	to strong resolvent convergence of the subdifferentials, i.e., to (a).
\end{proof}

\begin{rem}\label{rem:moscosym}
	The implication from (b) to (a) in Theorem~\ref{thm:strongconv} remains valid for symmetric, but not necessarily positive forms provided that there
	exists $\omega \le 0$ such that $a_n(u,u) - \omega \|j_n(u)\|_H^2 \ge 0$ for all $u \in V_n$ and
	$a(u,u) - \omega \|j(u)\|_H^2 \ge 0$ for all $u \in V$.
	In fact, assume that the conditions in (b) are fulfilled. Lower semicontinuity of the norm in $H$ yields
	that then also the positive forms $\tilde{a}_n$ and $\tilde{a}$ given by $\tilde{a}_n(u,v) \coloneqq a_n(u,v) - \omega {\scalar{j_n(u)}{j_n(v)}}_H$
	and $\tilde{a}(u,v) \coloneqq a(u,v) - \omega {\scalar{j(u)}{j(v)}}_H$ satisfy the conditions in (b).
	Now the theorem implies that the associated operators $(-A_n - \omega)$ converge to $(-A - \omega)$ in the strong
	resolvent sense, which trivially implies (a).
\end{rem}

Let $H_1$ and $H_2$ be separable Hilbert spaces.
For $p\in [1,\infty)$ the $p$-Schatten class is defined by
$${\mathcal L}_p(H_1,H_2):=\{T\in{\mathcal K}(H_1,H_2):\|T\|_{{\mathcal L}_p}:=\|(s_n)_{n\in\mathbb N}\|_{\ell^p}<\infty\},$$
where $(s_n)_{n\in\mathbb N}$ is the sequence of \emph{singular values} of $T$, i.e.,
the sequence of eigenvalues of $|T| \coloneqq (T^*T)^\frac12$. Then
$\|\cdot\|_{{\mathcal L}_p}$ is a complete norm on $\mathcal{L}_p(H_1,H_2)$, called the $p$-Schatten norm.
The operators in $\mathcal{L}_1(H_1,H_2)$ are also calls \emph{trace class operators} with the
\emph{trace norm}, and the operators in $\mathcal{L}_2(H_1,H_2)$ are called \emph{Hilbert--Schmidt operators}.
If $H_1 = H_2 = H$ we frequently write $\mathcal{L}_p(H)$ instead of $\mathcal{L}_p(H,H)$.
For more information about the Schatten classes we refer to~\cite{GohKre69,Sim05}.

We are mainly interested in semigroups consisting of Schatten class operators.
The following definition goes back to Dietrich A.\ Uhlenbrock~\cite{Uhl71} and first appeared
in applications in statistical mechanics.
Nowadays, Gibbs semigroups are popular objects in mathematical physics.

\begin{defi}
Let $H$ be a Hilbert space. A \emph{Gibbs semigroups} is a $C_0$-semigroup
$(T(t))_{t\ge 0}$ on $H$ such that each operator $T(t)$, $t>0$, is of trace class.
\end{defi}

\begin{rems}\label{versch}
\begin{enumerate}[(1)]
\item
Since
$\mathcal{L}_p(H) \cdot \mathcal{L}_{q}(H) \subset \mathcal{L}_r(H) \subset \mathcal{L}_{r'}(H)$ for $\frac{1}{r} = \frac{1}{p} + \frac{1}{q}$
and $r < r'$, every semigroup $(T(t))_{t \ge 0}$ for which there exists $p \in [1,\infty)$ such that $T(t) \in \mathcal{L}_p(H)$ for all $t > 0$
is a Gibbs semigroup. 
\item
Let $X$ be a finite measure space. It is known that each bounded linear
operator $T$ from $L^2(X)$ to $L^\infty(X)$ is a Hilbert--Schmidt
operator~\cite[Thm.~1.6.2]{Are06}. In particular every ultra-contractive semigroup on $L^2(X)$ is a
Gibbs semigroup. Hence Proposition~\ref{prop:ultra} provides a sufficient condition
for the Gibbs property, which is sometimes easy to check.
\item It seems to be difficult to characterise the Gibbs property in terms of the resolvent.
If $-A$ generates an analytic semigroup $(T(t))_{t \ge 0}$ on $H$ and $(\lambda + A)^{-k} \in \mathcal{L}_p(H)$
for some $k \in \mathbb{N}$, some $\lambda$ in the resolvent set and some $p \in [1,\infty)$, then $(T(t))_{t \ge 0}$ is a Gibbs
semigroup. In fact, since in that case the embedding
$D(A^k) \hookrightarrow H$ is of Schatten class and $T(t)\colon H \to D(A^k)$ is bounded for $t > 0$,
the ideal property implies that $T(t) \in \mathcal{L}_p(H)$ for all $t > 0$.

But the converse fails. In fact, consider the diagonal operator $A = D_\lambda$
on $\ell^2$ and $(T(t))_{t \ge 0} = (e^{-tA})_{t \ge 0}$, where $\lambda_n
\coloneqq \log^2 n$. Then the eigenvalues $e^{-t \log^2n} = n^{-t \log n}$ of
$T(t)$ are summable for every $t > 0$, i.e., $(T(t))_{t \ge 0}$ is a Gibbs
semigroup, but the eigenvalues $(\lambda+\log^2n)^{-k}$ of $(\lambda+A)^{-k}$ are
not $p$-summable for any $k \in \mathbb{N}$, $p \in [1,\infty)$ and $\lambda$ in the resolvent set.

\item The square root of the above operator $D_\lambda$ yields also another
interesting counterexample. It is known that for an analytic semigroup
immediate compactness and eventual compactness are equivalent. However, the
square root of $D_\lambda$ generates a semigroup whose eigenvalues $e^{-t \log
n} = n^{-t}$ are $p$-summable if and only if $t > 1/p$. In particular, this self-adjoint semigroup
is eventually Gibbs, but not immediately Gibbs.

\item
It is known that for a bounded domain $\Omega\subset {\mathbb R}^d$ with the
cone property the embedding of $H^k(\Omega)$ into $L^2(\Omega)$ is a
Hilbert--Schmidt operator whenever $2k > d$, see~\cite{Mau61}, and in fact a
$p$-Schatten class operator if $pk>d$, see~\cite{Gra68}. Under certain
assumptions on the geometry, Maurin's and Gramsch's result have been extended
to unbounded domains~\cite{Cla66,Koe77}. In such situations, if $A$ generates
an analytic semigroup and $D(A) \subset H^1(\Omega)$, then $A$ generates a
Gibbs semigroup. Observe that by~\cite[Thm.~6.54 and Rem.~6.55]{Ada75} there
exist domains with infinite measure such that the embedding of $H^1(\Omega)$
into $L^2(\Omega)$ is in $\mathcal{L}_p(H^1(\Omega),L^2(\Omega))$ for some $p
\in [1,\infty)$. In this situation criterion (2) does not apply.
\item
The preceding criterion can also be useful for semigroups on
Sobolev spaces $H^s$ with index $s\not=0$.
For example, it allows us to prove that the semigroup
generated by the Wentzell--Robin Laplacian on a smooth domain (see Section~\ref{sec:wro} for
details) on $H^1(\Omega)$ considered in~\cite[\S 2.9]{AreMetPal03}
and~\cite{FavGolGol03} is Gibbs. To be more precise, recall that the domain of
the Wentzell--Robin-Laplacian is a subspace of $H^\frac{3}{2}(\Omega)\times
L^2(\partial \Omega)$. Thus, the semigroup generated by its part in
$V \coloneqq \{ (u, u_{|\partial\Omega}) : u \in H^1(\Omega) \}$ maps $V$ to $\{ (u,u_{|\partial\Omega}) : u \in H^\frac{3}{2}(\Omega) \}$
for all $t > 0$.
By~\cite[Satz 1]{Gra68} the embedding $H^\frac{3}{2}(\Omega)\hookrightarrow
H^1(\Omega)$ is a $p$-Schatten class operator for all $p>2d$, hence so is any
operator of the semigroup for $t>0$, and by part (1) this
semigroup is Gibbs.
The same argument applies to general (also non-selfadjoint) elliptic operators with Wentzell-Robin or
similar boundary conditions.
On the other hand, part (2) does not yield the result in this case since the semigroup is not defined on an $L^2$-space.
\end{enumerate}
\end{rems}

We apply known result about convergence in Schatten norms, cf.~\cite[Chapter 2]{Sim05},
to semigroups arising from $j$-elliptic forms.
The following proposition is a direct consequence of Proposition~\ref{prop:dom} together with~\cite[Lemma, p.271]{Zag80}.
Its conditions are often easy to check; we will give some examples later on.

\begin{theo}\label{thm:semiconv}
Let $H$ be a Hilbert space and let $(a_n,j_n)$, $(a,j)$ and $(b,j)$ be symmetric, sesquilinear
forms that satisfy the conditions in Definition~\ref{def:complete}.
We denote by $A_n$, $A$ and $B$ the associated self-adjoint operators.
Assume that
\begin{enumerate}[(i)]
\item $(b,j) \le (a_n,j_n)$ for all $n \in \mathbb{N}$ in the sense of Proposition~\ref{prop:dom},
\item $-B$ generates a Gibbs semigroup, and
\item $(A_n)$ converges to $A$ in the strong resolvent sense.
\end{enumerate}
Then 
$$\lim_{n\to \infty}e^{-tA_n}= e^{-tA}\qquad \hbox{in }{\mathcal L}_1(H)$$
for every $t>0$.
\end{theo}

\begin{rem}
Theorem~\ref{thm:semiconv} tells us that the existence of a dominating form
implies trace norm convergence of the semigroup. This is
remarkable because, even though form domination implies domination for the
resolvents, it does in general not imply domination for the
semigroups. In fact, for $A \coloneqq
\left(\begin{smallmatrix} 2 & 2 \\ 2 & 2 \end{smallmatrix}\right)$ and $B
\coloneqq \left(\begin{smallmatrix} 3 & 0 \\ 0 & 6 \end{smallmatrix}\right)$ we
have $0 \le A \le B$, but $e^{-B} \not\le e^{-A}$ in the sense of positive definiteness.
The authors are grateful to Ulrich Groh (T\"ubingen) for pointing out this example.
\end{rem}

Finally, we also consider convergence of semigroups in Schatten norms as operators
between different Hilbert spaces. We obtain our main result as a consequence of an interpolation
theorem for Schatten class operators. A criterion which enables us to check the trace norm
convergence required in the following theorem was given in Theorem~\ref{thm:semiconv}. 

\begin{theo}\label{thm:sobschatt}
Let $p\in [1,\infty)$. Let $(A_n),A$ be uniformly m-sectorial operators on $H$,
i.e., m-sectorial operators with uniform constants, which generate Gibbs semigroups. 
Assume that there exists a subspace $\tilde{H}$ of $H$ such that
\begin{itemize}
\item $\tilde{H}$ is a Hilbert space,
\item $\tilde{H}$ is compactly embedded in $H$, and
\item there exists some $k \in \mathbb{N}$ such that $D(A_n^k) \subset \tilde{H}$ for all $n \in \mathbb{N}$ with uniform embedding constants.
\end{itemize}
If
$$\lim_{n\to \infty}e^{-tA_n}= e^{-tA}\qquad \hbox{in }\mathcal{L}_p(H)$$ 
for every $t>0$, then 
$$\lim_{n\to \infty}e^{-tA_n} = e^{-tA}\qquad \hbox{in }\mathcal{L}_q(H,H_\theta)$$ 
for every $t>0$ and every $\theta\in (0,1)$, where $q \in [1,\infty)$ is given by
\[
	\frac{1}{q} = \frac{\theta}{p} + (1-\theta)
\]
and where $H_\theta$ denotes the complex interpolation space
$[\tilde{H},H]_\theta$.
\end{theo}

\begin{proof}
We first show that $D(A^k)$ is continuously embedded into $\tilde{H}$.
Fix $\lambda > 0$ so large that $\lambda + A_n$ is invertible with uniformly bounded inverse
with respect to $n \in \mathbb{N}$. Take $u\in H$. Then the uniform constants in the m-sectoriality and the embeddings ensure that
the sequence $((\lambda+A_n)^{-k} u)_{n\in\mathbb N}$ is bounded in $\tilde{H}$. Hence there exists a weakly convergent subsequence in
$\tilde{H}$, which necessarily converges to $(\lambda+A)^{-k} u$ since the semigroups and hence the resolvents
converge strongly by assumption. This proves that
$D(A^k)\subset \tilde{H}$. Now the closed graph theorem yields $D(A^k)\hookrightarrow \tilde{H}$.

For every $t>0$ and every $n\in\mathbb N$ the operator $e^{-tA_n}=e^{-\frac{t}{2} A_n}e^{-\frac{t}{2} A_n}$
is a composition of an operator in $\mathcal{L}_1(H)$
and an operator in $\mathcal{L}(H,\tilde{H})$, both with uniformly estimable norms, compare (1) in Remarks~\ref{versch}.
Hence $\sup_{n \in \mathbb{N}} \|e^{-tA_n}\|_{\mathcal{L}_1(H,\tilde{H})} < \infty$ by the ideal property of the norm, and by a similar argument
$e^{-tA}$ is in $\mathcal{L}_1(H,\tilde{H})$ as well.

Now we obtain from an interpolation result for Schatten class operators~\cite{Gap74} that
\[
\|e^{-tA_n}- e^{-tA}\|_{{\mathcal L}_q(H,H_\theta)}\le C\|e^{-tA_n}- e^{-tA}\|^\theta_{{\mathcal L}_1(H,\tilde{H})}\|e^{-tA_n}- e^{-tA}\|^{1-\theta}_{{\mathcal L}_p(H)},
\]
for some constant $C \ge 1$ since the fractional domain space considered in~\cite{Gap74}
coincides with $H_\theta$ up to equivalent norms.
The first factor is bounded by the above considerations whereas the second factor converges to zero by assumption.
\end{proof}

Let us combine several of our observations into a final result.
\begin{cor}\label{cor:summarize}
	Let $X$ be a finite measure space.
	Let $(a_n,j_n)$, $(a,j)$ and $(b,j')$ be positive elltiptic forms on $L^2(X)$ in the sense of Definition~\ref{def:complete}
	with form domain $V$, and denote the associated self-adjoint operators by $A_n$, $A$ and $B$, respectively.
	Let $\tilde{H}$ be a dense subspace of $H$, which is a Hilbert space in its own right.
	Assume that
	\begin{itemize}
	\item
		$(a_n,j_n)$ converges to $(a,j)$ in the sense of Mosco;
	\item
		$(b,j') \le (a_n,j_n)$ for all $n \in \mathbb{N}$;
	\item
		there exists $q > 2$ such that $j'(V) \subset L^q(X)$;
	\item
	for all $u\in V$ there exists $w\in V$ such that
		$(|j(u)| \wedge 1)\sgn j(u)=j(w) $ and $\Real b(w, u-w) \ge 0$;
	\item
		$\tilde{H}$ is compactly embedded into $H$;
	\item
		$D(A_n^k) \subset \tilde{H}$ for some $k \in \mathbb{N}$ with an embedding
		constant that is uniform in $n \in \mathbb{N}$;
	\end{itemize}
	For arbitrary $\theta \in (0,1)$ let $H_\theta$ denote the complex interpolation space $H_\theta = [\tilde{H},H]_\theta$.
	Then $e^{-tA_n} \to e^{-tA}$ in the trace norm $\mathcal{L}_1(H,H_\theta)$ and hence in particular
	in the operator norm $\mathcal{L}(H,H_\theta)$.
\end{cor}

\begin{proof}
By the invariance criterion for $j$-elliptic forms~\cite[Prop~2.9]{AreEls09} the semigroup is $L^\infty(X)$-contractive, analogously to the situation in~\cite[Thm.~2.13]{Ouh05}. Hence by Proposition~\ref{prop:ultra}
	it is ultra-contractive and thus Gibbs by Remark~\ref{versch}. Since in addition $(-A_n)$ converges to $-A$
	in the strong resolvent sense by Theorem~\ref{thm:strongconv} we obtain from Theorem~\ref{thm:semiconv}
	that $e^{-tA_n} \to e^{-tA}$ in $\mathcal{L}_1(H)$. Now the assertion follows from Theorem~\ref{thm:sobschatt}.
\end{proof}

\begin{rem}
	The assumption of Corollary~\ref{cor:summarize} that $D(A_n^k) \subset \tilde{H}$ with uniform embeddings
	is in particular satisfied for $\tilde{H} = V$ if the constants in the ellipticity estimate~\eqref{eq:elliptic} of $(a_n,j_n)$ are
	uniform in $n$, for the semigroups $(e^{tA_n})$ are bounded as operators from $H$ to $V$, uniformly in $n$. 
	
	We emphasise that in this special case Corollary~\ref{cor:summarize} yields a convergence result for semigroups under assumptions solely on the associated forms, with no reference to the associated operators.
\end{rem}

\begin{rems}
Let us finally remark on the Gibbs property for other kinds of operator families.
\begin{enumerate}[(1)]
\item Let $-A$ be a self-adjoint operator, hence the generator of a sine
operator function $(S(t))_{t \in \mathbb{R}}$, cf.~\cite[\S~3.15]{AreBatHie01}. It is known  that $S(t)$
maps $H$ into $V$ for all $t\in\mathbb R$, where $V$ is the domain of the form
associated with $A$. If the embedding of $V$ into $H$ is of $p$-Schatten class
(e.g., $V$ a closed subspace of $H^1(0,1)$, $H = L^2(0,1)$ and $p > 1$, cf.~Remark~\ref{versch}.(5)),
then $S(t)$ is of $p$-Schatten class for all $t\in\mathbb R$.
\item
Unlike in the semigroup case, however, there exist sine operator functions $(S(t))_{t \in \mathbb{R}}$ on
a Hilbert space $H$ such that $S(t) \in \mathcal{L}_p(H)$ for all $t \in \mathbb{R}$ for some $p > 1$,
but $S(t) \not\in \mathcal{L}_{p-\eps}(H)$ for all $t \in \mathbb{R}$ and all $\eps > 0$.
In fact, fix $\alpha\ge 1$ and consider the multiplication
operator $M_\lambda$ on $\ell^2$, where the sequence $\lambda$ is
given by $$\lambda_n:=-\left(\frac{\pi}{2}+2\pi \lfloor n^\alpha\rfloor
\right)^2,\qquad n\in\mathbb N$$ and $\lfloor x\rfloor$ denotes the greatest
integer below $x$. Then the corresponding sine operator function is given by
$$S(t)x:=\bigl(\tfrac{\sinh(\sqrt{\lambda_n} t)}{\sqrt{\lambda_n}}x_n\bigr)_{n\in\mathbb N}
	=\bigl(\tfrac{\sin((\frac{\pi}{2} + 2\pi\lfloor n^\alpha \rfloor)t)}{\frac{\pi}{2} + 2\pi \lfloor n^\alpha \rfloor} x_n\bigr)_{n \in \mathbb{N}},\quad t\in\mathbb R,\;x\in \ell^2,$$
so $S(t) \in \mathcal{L}_p(\ell^2)$ 
for all $p>\alpha^{-1}$ and all $t\in\mathbb R$, but $S(1)$ is not in
${\mathcal L}_{\alpha^{-1}}(\ell^2)$.

\item On an infinite dimensional Hilbert space a cosine operator function with unbounded generator never
consists of Schatten class operators on a non-void open interval. In fact, a
cosine operator function can only be compact on an interval of positive length
if its generator is a bounded operator, cf.~\cite[Lemma~2.1]{TraWeb77}.
\end{enumerate}

\end{rems}

\section{Applications}\label{application:sec}

\subsection{Convergence of Laplacians with respect to higher regularity Schatten norms}
We begin with an application of our result about Schatten convergence, which shows
how our convergence results can be combined to treat semigroups generated by elliptic operators:
starting with convergence in the strong resolvent sense we are able
to obtain trace norm convergence with respect to Sobolev spaces of
arbitrarily high order.

\begin{theo}
Let $\Omega$ be a bounded open
domain in $\mathbb{R}^d$ with $C^\infty$-boundary. Consider a sequence of Laplacians
$\Delta_{k_n}$ with Robin boundary conditions
$$\frac{\partial u}{\partial \nu}+k_n u=0\qquad \hbox{on }\partial \Omega.$$
for constants $(k_n)_{n\in\mathbb N}\subset [0,\infty)$. If
$(k_n)$ is a monotonically decreasing null sequence, then
$$\lim_{n\to \infty}e^{t \Delta_{k_n}} = e^{t \Delta_N}\qquad \hbox{in }\mathcal{L}_1(L^2(\Omega),H^\ell(\Omega))$$
for every $t > 0$ and every $\ell \in \mathbb{N}$, where $\Delta_N$ denotes the
Laplace operator on $\Omega$ with Neumann boundary conditions.
\end{theo}

\begin{proof}
By~\cite[Thm. 8.3.11]{Kat95} the sequence $(\Delta_{k_n})$ converges to $\Delta_N$
in the strong resolvent sense. Let $(a_n)$ and $a_N$ be the elliptic classical forms associated with
$-\Delta_{k_n}$ and $-\Delta_N$, respectively. Then $a_n \ge a_N$ in the sense of Proposition~\ref{prop:dom}.
Moreover, $\Delta_N$ generates a Gibbs semigroup, see (2) in Remarks~\ref{versch} and use~\cite[Corollary~2.17 and Theorem~6.4]{Ouh05}.
Hence $e^{-t\Delta_{k_n}} \to e^{-t\Delta_N}$ in $\mathcal{L}_1(L^2(\Omega))$
by Theorem~\ref{thm:semiconv}. Moreover, following the proofs of elliptic regularity, cf.~\cite[\S 2.5.1]{Gri85},
one can see that $D(\Delta_{k_n}^\ell)$ is uniformly embedded into $H^{2\ell}(\Omega)$ for every $\ell \in \mathbb{N}$.
Applying Theorem~\ref{thm:sobschatt} with $\theta = \frac12$
we conclude that $e^{t \Delta_{k_n}} \to e^{t \Delta_N}$ in $\mathcal{L}_1(L^2(\Omega),H^\ell(\Omega))$
for every $t > 0$ and every $\ell \in \mathbb{N}$.
\end{proof}

\begin{rem}
Analogous arguments work for heat equations with the dynamic boundary conditions
$$\frac{\partial u}{\partial t}=-\frac{\partial u}{\partial \nu}-k_n u\qquad \hbox{on }\partial \Omega,$$
which arise from forms as seen in~\cite{AreMetPal03}.
This complements the results of~\cite{CocGolGol08}, where the emphasis
lies in obtaining sharp estimates for the rate of convergence with
respect to the $H^1$-operator norm.
\end{rem}

\subsection{Convergence of Laplacians with variable boundary conditions on exterior domains}\label{sec:birapp}

The result in this section is somewhat special, since we prove Schatten norm
convergence of diffusion semigroups $(T_n(t))_{t \ge 0}$  to a semigroup $(T(t))_{t \ge 0}$, all acting on spaces of functions on exterior domains with varying boundary conditions. As we will see, in this situation it is sometimes possible to
obtain that $T_n(t) - T(t) \to 0$ in $\mathcal{L}_p$ (for sufficiently large
values of $p$) as $n\to \infty$ even though the operators $T_n(t)$ and $T(t)$ are not individually in
$\mathcal{L}_p$ and in fact not even compact. In particular, Theorem~\ref{thm:semiconv} does not apply here. Instead, our argument relies upon classical results on differences of differential operators first due to Mikhail \v{S}.\ Birman~\cite[Thm.~3.8]{Bir08} and recently improved in~\cite{BehLanLob10}; only Theorem~\ref{thm:strongconv} is additionally needed. 
Such situations indeed appear frequently in mathematical physics, see for example~\cite{Sto94,DemStoSto95,CheExnTur10}.

\begin{theo}\label{thm:birm1}
Let $\Omega\subset{\mathbb R}^d$, $d \ge 3$, be an exterior domain with smooth
boundary and $\Delta_\beta$ the Laplace operator on $\Omega$ with Robin boundary condition
\[
	\frac{\partial u}{\partial \nu}+\beta u = 0\qquad \hbox{on }\partial \Omega.
\]
If $(\beta_n)_{n\in\mathbb N}$ is a bounded sequence in $L^\infty(\partial\Omega)$ and converges to a function $\beta_0$ almost everywhere, then 
$$\lim_{n\to \infty} \bigl( e^{t\Delta_{\beta_n}} - e^{t\Delta_{\beta_0}} \bigr) = 0 \qquad \hbox{in }{\mathcal L}_p(L^2(\Omega))$$ 
for every $t>0$ and all $p>\frac{d-1}{3}$.
\end{theo}

\begin{proof}
Let us first show that the operators are uniformly m-sectorial.
Since the trace operator $u \mapsto u|_{\partial\Omega}$ is compact from $H^1(\Omega)$ to $L^2(\partial\Omega)$,
by Lemma~\ref{lem:eberlein} there exists
$c > 0$ such that
\[
	\|u\|_{L^2(\partial\Omega)}^2 \le \frac{1}{2} \|\nabla u\|_{L^2(\Omega)}^2 + c \|u\|_{L^2(\Omega)}^2
\]
for all $u \in H^1(\Omega)$. This shows that the quadratic form $q_\beta$ associated with $-\Delta_\beta$, i.e.
\[
	q_\beta(u) \coloneqq \int_\Omega |\nabla u|^2 + \int_{\partial\Omega} \beta_n |u|^2
\]
for $u \in H^1(\Omega)$, is semi-bounded for every essentially bounded function $\beta$
and hence that $\Delta_\beta$ generates a $\mathrm{C}_0$-semigroup on $L^2(\Omega)$.
More precisely,
\[
	q_{\beta_n}(u) \ge \frac{1}{2} \int_\Omega |\nabla u|^2 - c m \|u\|_{L^2(\Omega)}^2
\]
with $m \coloneqq \sup_{n \in \mathbb{N}} \|\beta_n\|_\infty$. This proves uniform m-sectoriality.

Next we show convergence in the strong resolvent sense.
We write $\tilde{q}_\beta(u) \coloneqq q_\beta(u) + (c m + 1) \|u\|_{L^2(\Omega)}^2$ for simplicity
and show that $\tilde{q}_{\beta_n}$ converges to $\tilde{q}_{\beta_0} + c m + 1$ in the sense of Mosco.
To this end, let $(u_n)$ be a sequence in $H^1(\Omega)$ such that $u_n \rightharpoonup u$ in $L^2(\Omega)$
and $\liminf_{n \to \infty} \tilde{q}_{\beta_n}(u_n) < \infty$.
Then $(u_n)$ is bounded in $H^1(\Omega)$, hence passing to a subsequence we can assume that $u_n \rightharpoonup u$ in $H^1(\Omega)$.
Thus in particular $u \in H^1(\Omega)$.
Now by compactness $u_n|_{\partial\Omega} \to u|_{\partial\Omega}$ in
$L^2(\partial\Omega)$, so in particular
\[
	\int_{\partial\Omega} \beta_n |u_n|^2 \to \int_{\partial\Omega} \beta_0 |u|^2.
\]
By weak lower semicontinuity of the norm of $H^1(\Omega)$
this proves
\[
	\tilde{q}_{\beta_0}(u) \le \liminf_{n \to \infty} \tilde{q}_{\beta_n}(u_n).
\]
Moreover, for given $u \in H^1(\Omega)$ we clearly have $\tilde{q}_{\beta_n}(u) \to \tilde{q}_{\beta_0}(u)$ by Lebesgue's
theorem. We thus have shown that $\Delta_{\beta_n} \to \Delta_{\beta_0}$ in the strong resolvent sense, see Theorem~\ref{thm:strongconv}.

We now proves the convergence in Schatten norm.
Since $-\Delta_{-m} \le -\Delta_{\beta_n} \le -\Delta_m$
in the form sense we have
\[
	(c m + 1 - \Delta_m)^{-1} \ge (c m + 1 - \Delta_{\beta_n})^{-1} \ge (c m + 1 - \Delta_{-m})^{-1}
\]
as self-adjoint operators~\cite[Thm.~2.21]{Kat95}. A similar assertion holds for $\Delta_{\beta_0}$. Consequently,
\[
	|(c m + 1 - \Delta_{\beta_n})^{-1} - (c m + 1 - \Delta_{\beta_0})^{-1}|
		\le  (c m + 1 - \Delta_m)^{-1} - (c m + 1 - \Delta_{-m})^{-1},
\]
where the right hand side is in $\mathcal{L}_p(L^2(\Omega))$
for every $p > \frac{d-1}{3}$ by~\cite[Cor.~3.6]{BehLanLob10}.
This implies that $(c m + 1 - \Delta_{\beta_n})^{-1}$ converges to $(c m + 1 - \Delta_{\beta_0})^{-1}$
in $\mathcal{L}_p(L^2(\Omega))$ as $n \to \infty$, see~\cite[Prop.~2.1]{NeiZag90b}.

Moreover, as in the proof of~\cite[Thm.~3.5]{BehLanLob10}, for all $\lambda$ and $\mu$ in the sector $\Sigma \coloneqq \mathbb{C} \setminus \mathbb{R}_{\le cm}$ we have
\begin{align*}
	& (\mu - \Delta_{\beta_n})^{-1} - (\mu - \Delta_{\beta_0})^{-1} \\
		& = \bigl( I + (\lambda - \mu) (\mu - \Delta_{\beta_0})^{-1} \bigr) \bigl( (\lambda - \Delta_{\beta_n})^{-1} - (\lambda - \Delta_{\beta_0})^{-1} \bigr) \bigl( I + (\lambda - \mu) (\mu - \Delta_{\beta_n})^{-1} \bigr).
\end{align*}
In fact, this identity is certainly satisfied on $V=H^1(\Omega)$ since $\kappa - \Delta_{\beta_n}$ and
$\kappa - \Delta_{\beta_0}$ are isomorphisms from $V$ to the dual space $V'$.
Thus the identity extends to $L^2(\Omega)$ by denseness.

Picking $\lambda = cm+1$ we obtain from the ideal property that
\[
	(\mu - \Delta_{\beta_n})^{-1} - (\mu - \Delta_\beta)^{-1} \in \mathcal{L}_p(L^2(\Omega))
\]
for all $\mu \in \Sigma$. More precisely we even obtain that on every sector smaller than $\Sigma$ this sequence
of differences is bounded and convergent in $\mathcal{L}_p(L^2(\Omega))$ on compact subsets of $\Sigma$,
uniformly with respect to $n$.
Here we have used that the operators $\Delta_{\beta_n}$ are uniformly m-sectorial.
Hence the integral representation~\cite[(3.46)]{AreBatHie01}
\[
	e^{t\Delta_{\beta_n}}-e^{t\Delta_\beta} =
		\frac{1}{2\pi i}\int_\gamma e^{-t\lambda} \left((\lambda-\Delta_{\beta_n})^{-1} -(\lambda-\Delta_\beta)^{-1}\right) d\lambda,
\]
shows that $e^{t\Delta_{\beta_n}}$
converges to $e^{t\Delta_\beta}$ in $\mathcal{L}_p(L^2(\Omega))$ for every $t > 0$.
\end{proof}

\subsection{Coupled boundary conditions}
In this subsection we consider convergence for systems of Laplacians
with a certain coupled boundary conditions, which are motivated by quantum graphs.
It seems that~\cite[Thm.~VI.3.6]{Kat95} cannot be used to obtain strong convergence of
the resolvents in this example, so we employ an approach developed by Olaf Post~\cite{Pos06} instead,
where we use the notation of~\cite{MugNitPos10}.

\begin{theo}\label{olafrev}
Let $(Y_n)_{n\in\mathbb N}$ be a sequence of closed subspaces of $\mathbb C^k$,
$k\in\mathbb N$. Let $\Omega\subset{\mathbb R}^d$ be an exterior domain with
smooth compact boundary and $(\Delta_{Y_n})_{n\in\mathbb N}$ be a sequence of
Laplacians on $L^2(\Omega;{\mathbb C}^k)$ with boundary conditions
$$u|_{\partial \Omega}\in Y_n\qquad\hbox{and}\qquad \frac{\partial u}{\partial
\nu}\in Y_n^\perp\quad a.e.,\;n\in\mathbb N.$$ Assume that
there exist a subspace $Y$ of ${\mathbb C}^k$ and a family $({J}^{\downarrow n})_{n\in\mathbb
N}$ of unitary operators on $H$ converging to the identity $I$ such that ${J}^{\downarrow n}Y_n=Y$ for all
$n\in\mathbb N$. Denote by $\Delta_Y$ the Laplacian with corresponding
boundary conditions. 
Then 
$$\lim_{n\to \infty}e^{t\Delta_{Y_n}}=e^{t\Delta_Y}\qquad \hbox{ in }\mathcal L_p(L^2(\Omega;{\mathbb C}^k))$$ 
for every $t>0$ and all $p>\frac{d-1}{2}$.
\end{theo}

\begin{proof}
We introduce elliptic forms $(a_n)_{n\in\mathbb N}$ and $a_0$ with form domains
\begin{align*}
	V_n & :=\{f\in H^1(\Omega;{\mathbb C}^k):f|_{\partial\Omega}\in Y_n\} \qquad n\in\mathbb N, \\
	V & :=\{f\in H^1(\Omega;{\mathbb C}^k):f|_{\partial\Omega}\in Y\}
\end{align*}
as in~\cite[\S 2]{CarMug09}. These forms are symmetric, and accordingly the
associated Laplacians $\Delta_{Y_n}$ and $\Delta_Y$ are self-adjoint operators
on $H:=L^2(\Omega;{\mathbb C}^k)$.

We set $${\mathcal J}^{\downarrow n}
f:={\mathcal J}^{\downarrow n}_{1} f:=\Jndown\circ f\qquad \hbox{and}\qquad
{\mathcal J}^{\uparrow n} f:={\mathcal J}^{\uparrow_ n}_1 f:=\Jnup\circ
f,\qquad n\in\mathbb N.$$ Since these operators are unitary on
$H=L^2(\Omega;{\mathbb C}^k)$ as well as from $V_n$ to $V$, it is easy to see
that the assumptions in~\cite[Def.~2.3]{MugNitPos10} are satisfied, and we
deduce from~\cite[Prop.~3.4]{MugNitPos10} that $\Delta_{Y_n}$ converges to $\Delta_Y$
in the norm resolvent sense.

Moreover, $-\Delta_{\mathbb{C}^k} \le -\Delta_{Y_n} \le -\Delta_{\lbrace 0 \rbrace}$ in the form sense and
hence
\[
	| (\lambda - \Delta_{Y_n})^{-1} - (\lambda - \Delta_{Y_0})^{-1} | \le (\lambda - \Delta_{\lbrace0\rbrace})^{-1} - (\lambda - \Delta_{\mathbb{C}^k})^{-1}
\]
by~\cite[Thm.~2.21]{Kat95}. Using that $\Delta_{\lbrace0\rbrace}$ and  $\Delta_{\mathbb{C}^k}$
act as uncoupled copies of $k$ Dirichlet and Neumann Laplace operators, respectively,
we obtain from Birman's result~\cite[Thm.~3.8]{Bir08} that the operator on the right hand
side is in $\mathcal{L}_p(L^2(\Omega))$ for every $p > \frac{d-1}{2}$. The conclusion
now follows as in Theorem~\ref{thm:birm1}.
\end{proof}

\begin{rem}
We have formulated Theorems~\ref{thm:birm1} and~\ref{olafrev} in the case of
Laplacians only for the sake of simplicity: in fact, both results can be extended to
strongly elliptic operators with coefficients in
$W^{1,\infty}(\Omega)$. Also, rougher and even non-compact
boundaries can be allowed, leading to convergence only for
larger $p$, cf.~\cite[Rem.~3.4, Thm.~3.8 and Thm.~5.2]{Bir08}.
\end{rem}

\subsection{Dirichlet-to-Neumann-type operators}\label{sec:dir-to-neu}
By showing that the Dirichlet-to-Neumann operator is associated with a
$j$-elliptic form, Arendt and ter Elst have delivered a most interesting
application of their theory. This is an instance where a non-injective $j$
appears in a natural way.

Consider an open bounded domain $\Omega\subset{\mathbb R}^d$
with Lipschitz boundary, where $d \ge 2$, and let $V \coloneqq H^1(\Omega)$ and $H \coloneqq L^2(\partial\Omega)$.
We consider the sesquilinear form $a$ defined by
\[
a(u,v):=  \int_\Omega \alpha \nabla u\cdot  \overline{\nabla v},\qquad u,v\in V,
\]
where the matrix-valued coefficient $\alpha\in L^\infty(\Omega;\mathbb C^{d \times d})$ is uniformly positive definite,
i.e., for a.e.\ $x\in \Omega$ the matrix $\alpha(x)$ is Hermitian and satisfies
$$(\alpha(x)\xi|\xi)\ge k_0 |\xi|^2\qquad \hbox{ for all }\xi\in\mathbb C^d$$
for some $k_0>0$. 
Let $j$ be the trace operator from $V$ to $H$. It can be
checked as in~\cite[\S 4.4]{AreEls09} that $a$ is a $j$-elliptic symmetric form, and more precisely
\[
	a(u,u) - \omega \|j(u)\|_H^2 \ge \mu \|u\|_V^2\qquad \hbox{for all }u\in V
\]
for some $\omega \in \mathbb{R}$ and $\mu > 0$.

Let $\lambda_1^D(a)$ denote the smallest eigenvalue of the operator associated with
the restriction of $a$ to to $H^1_0(\Omega)$, i.e.,
\[
	\lambda_1^D(a) \coloneqq \inf_{u \in H^1_0(\Omega)} \frac{a(u,u)}{\|u\|_{L^2(\Omega)}^2}.
\]
Let $\gamma_0 < \lambda_1^D(a)$ and fix a complex-valued function $\gamma \in L^q(\Omega)$,
$q > \frac{d}{2}$, satisfying $\Real \gamma \ge -\gamma_0$.
Then
\[
	b(u,v):=  \int_{\Omega} \gamma u \overline{v},\qquad u,v\in V,
\]
defines a bounded sesquilinear form on $V$ by the H\"older inequality and the Sobolev embedding theorem.

\begin{prop}
	Under the above assumptions, $a+b$ is $j$-elliptic.
\end{prop}
\begin{proof}
By the variational characterisation of $\lambda_1^D$ we have
for all $u \in H^1_0(\Omega)$ that
\[
	\Real b(u,u) \ge -\gamma_0 \int_\Omega |u|^2 \ge -\frac{\gamma_0}{\lambda_1^D(a)} a(u,u)
\]
and hence
\begin{equation}\label{eq:DtNfirstpart}
	a(u,u) + \Real b(u,u) \ge \tilde{\eta} a(u,u) \ge \tilde{\eta} k_0 \int_\Omega |\nabla u|^2
		\ge \eta \|u\|_V^2
\end{equation}
for all $u \in H^1_0(\Omega)$, where $\tilde{\eta} \coloneqq 1 - \frac{\gamma_0}{\lambda_1^D(a)} > 0$
and $\eta > 0$ depends on the first eigenvalue of the Dirichlet Laplacian on $\Omega$.
Moreover, $H^1(\Omega)$ is compactly embedded into $L^2(\Omega)$ and $j$ is injective on $V(a)$,
hence for all $u \in V(a)$ we have
\[
	\Real b(u,u) \ge -\gamma_0 \int_\Omega |u|^2
		\ge \frac{\mu}{2} \|u\|_V^2 - c_\mu \|j(u)\|_H^2
\]
for some $c_\mu \ge 0$ by Lemma~\ref{lem:eberlein} and hence
\begin{equation}\label{eq:DtNsecondpart}
	a(u,u) + \Real b(u,u) - (\omega - c_\mu) \|j(u)\|_H^2 \ge \frac{\mu}{2} \|u\|_V^2
\end{equation}
for all $u \in V(a)$.

For $u \in H^1_0(\Omega)$ and $v \in V(a)$ we have
$a(v,u) = a(u,v) = 0$ by definition of $V(a)$. Moreover, for every $\eps > 0$ we have
\begin{equation}\label{eq:DtNmixed}
	\begin{aligned}
		|b(u,v)| + |b(v,u)| & \le c \|u\|_{L^p(\Omega)} \|v\|_{L^p(\Omega)}
			\le \frac{c \eps}{2} \|u\|_{L^p(\Omega)}^2 + \frac{c}{2 \eps} \|v\|_{L^p(\Omega)}^2 \\
			& \le \frac{c \eps}{2} \|u\|_V^2 + c_\eps \|j(v)\|_H^2 + \eps \|v\|_V^2
	\end{aligned}
\end{equation}
for some $p \in [2,\frac{2(d-1)}{d-2})$ and some $c, c_\eps \ge 0$
by the integrability assumptions on $\gamma$, the Sobolev embeddings theorems
and Lemma~\ref{lem:eberlein}.

Since every $u \in V$ has a representation of the form $u = u_1 + u_2$
with $u_1 \in H^1_0(\Omega)$ and $u_2 \in V(a)$ by Remark~\ref{jclass},
combining~\eqref{eq:DtNfirstpart}, \eqref{eq:DtNsecondpart} and~\eqref{eq:DtNmixed},
where in the latter we pick $\eps > 0$ such that $\frac{c \eps}{2} < \frac{\eta}{2}$ and $\eps < \frac{\mu}{4}$,
we obtain that
\begin{align*}
	& a(u,u) + \Real b(u,u) \\ 
		& \quad = a(u_1,u_1) + \Real b(u_1,u_1) + a(u_2,u_2) + \Real b(u_2,u_2) + \Real b(u_1,u_2) + \Real b(u_2,u_1)\\
		& \quad \ge \eta \|u_1\|_V^2 + \frac{\mu}{2} \|u_2\|_V^2 + (\omega - c_\mu) \|j(u_2)\|_H^2 - \frac{\eta}{2} \|u_1\|_V^2 - \frac{\mu}{4} \|u_2\|_V^2 - c_\eps \|j(u_2)\|_H^2 \\
		& \quad = \frac{\eta}{2} \|u_1\|_V^2 + \frac{\mu}{4} \|u_2\|_V^2 + \omega' \|j(u_2)\|_H^2
\end{align*}
for some $\omega' \in \mathbb{R}$. Finally, since $V = H^1_0(\Omega) \oplus V(a)$ by Remark~\ref{jclass},
the expression $|||u|||^2 \coloneqq \|u_1\|_V^2 + \|u_2\|_V^2$ defines an equivalent norm on $V$, which allows
us to write the previous estimate as
\[
	a(u,u) + \Real b(u,u) - \omega' \|j(u)\|_H^2 \ge \mu' \|u\|_V^2\qquad \hbox{for all } u\in V,
\]
for some $\mu' > 0$. This is the $j$-ellipticity of $a+b$.
\end{proof}

Following~\cite[\S 4.4]{AreEls09} one can check that the
operator $-D^\gamma_{\alpha}$ associated with $a+b$ is some \emph{Dirichlet-to-Neumann} operator. 
More precisely, $\varphi \in L^2(\partial\Omega)$ is in $D(D^\gamma_{\alpha})$
if and only if there exists a (necessarily unique) weak solution of the inhomogeneous
Dirichlet problem
\begin{equation}\tag{IDP}
\left\{
\begin{aligned}
\gamma u - \operatorname{div}(\alpha\nabla u)&=0,\qquad &&x\in\Omega,\\
u(z)&=\varphi(z),\qquad &&z\in\partial\Omega,
\end{aligned}
\right.
\end{equation}
and the weak conormal derivative $\frac{\partial u}{\partial \nu_\alpha}$ exists as an element of
$L^2(\partial\Omega)$.
In this case, $-D_\alpha^\gamma u = \frac{\partial u}{\partial \nu_\alpha}$.
The above considerations show that $D_\alpha^\gamma$ generates an analytic semigroup.
We formulate this as a theorem.

\begin{theo}\label{theo:d-zu-n-gen}
Let $\Omega\subset{\mathbb R}^d$ be an open bounded domain with Lipschitz
boundary, where $d \ge 2$. Let $\alpha\in L^\infty(\Omega;\mathbb C^{d \times d})$ be uniformly positive definite,
and let $\gamma \in L^q(\Omega;{\mathbb C})$, $q > \frac{d}{2}$, be such that $\Real \gamma \ge -\gamma_0$ 
for some $\gamma_0 < \lambda_1^D$.
Then the operator $D^\gamma_{\alpha}$ generates an analytic semigroup on $H = L^2(\partial\Omega)$, which is Gibbs if additionally $\gamma\ge 0$.

If $\Omega$, $\alpha$ and $\gamma$ are smooth, then the analyticity angle of this semigroup is $\frac{\pi}{2}$.
\end{theo}

\begin{proof}
Let us prove the assertion on the analyticity angle. Let everything be smooth,
so that in particular $\gamma$ is bounded. By Proposition~\ref{lemma:crouz} it suffices to check that
$$M \|u\|_{H^1(\Omega)} \|u\|_{L^2(\partial\Omega)}\ge \left|\Ima \int_\Omega \gamma |u|^2 \right|\qquad\hbox{for all }u\in V(a+b)$$
holds for some $M \ge 0$, where $V(a+b)$ consists by definition of all
$H^1$-functions that are weak solutions of $\rm(IDP)$ for some $\gamma$.
Since $\Real \gamma \ge -\gamma_0$ the only function $u \in H^1_0(\Omega)$ satisfying $\gamma u - \operatorname{div}(\alpha \nabla u) = 0$
is $u=0$.
Hence by~\cite[Thm.~2.7.4]{LioMag72} the trace operator is an isomorphism from
$$\{ u\in H^\frac{1}{2}(\Omega): \gamma u - \operatorname{div}(\alpha \nabla u) = 0 \}$$
onto $L^2(\partial \Omega)$.
Accordingly, the estimate in Proposition~\ref{lemma:crouz} can be equivalently formulated as
$$M \|u\|_{H^1(\Omega)} \|u\|_{H^{\frac12}(\Omega)}\ge \left|\Ima \int_\Omega \gamma|u|^2 \right|\qquad\hbox{for all }u\in V(a+b)$$
for some possibly larger constant $M$. This is satisfied whenever $\gamma$ is bounded.

Assume now that $\gamma \ge 0$. Then by~\cite[Prop.~2.9]{AreEls09} the Dirichlet-to-Neumann semigroup of Theorem~\ref{theo:d-zu-n-gen}
submarkovian, i.e., positive and $L^\infty(\partial\Omega)$-contractive,
which is easily checked by a version of an invariance criterion due to Ouhabaz
for $j$-elliptic forms~\cite[Prop.~2.9]{AreEls09}, see also~\cite[Prop.~3.7]{AreEls10}.
In this case Proposition~\ref{prop:ultra} and the Sobolev embedding theorems for $\partial\Omega$
(see e.g.~\cite[Thm.~2.20]{Aub82}) yield in particular that
the Dirichlet-to-Neumann semigroup is a Gibbs semigroup.
\end{proof}

\begin{rem}
For the last step of the preceding proof we only need that $\gamma \in L^{\frac{2d}{3}}(\Omega)$.
Hence one could suspect
that for all such $\gamma$ the operator $D^\gamma_\alpha$ generates a cosine
operator function without any additional conditions on the smoothness of
$\alpha$, $\Omega$ and $\gamma$. However, to extend the result to this situation
we would need a generalisation of~\cite[Thm.~2.7.4]{LioMag72} to rough domains and rough
coefficients. A partial result into this direction is~\cite[Lemma.~3.1]{GesMit09},
where for the Laplace operator \cite[Thm.~2.7.4]{LioMag72} is extended to Lipschitz domains.
\end{rem}

\begin{rem}
We regard the perturbation we are considering as interesting mainly because
it cannot be expressed as a perturbation by an operator. In comparison,
if for smooth $\Omega$ we consider the vaguely related sesquilinear form $b':V\times V\to \mathbb C$ defined by
$b'(u,v):=  \int_{\partial\Omega} \beta u \overline{v}$
with $\beta \in L^{d-1}(\partial\Omega)$,
then $a + b'$ is associated with $-D^\gamma_\alpha - B$, where $B$ is a bounded operator
from $D(D^\gamma_\alpha)$ to $L^2(\Omega)$,
and we can deal with it using perturbation theorems for generators.
\end{rem}

\begin{rem}\label{rem:dtnSM}
The Gibbs property of the semigroup in Theorem~\ref{theo:d-zu-n-gen} has been observed before by Zagrebnov~\cite[Lemma~2.14]{Zag08}. 
His sketch of the proof is based on
a Weyl-type asymptotic result for the Dirichlet-to-Neumann
operator~\cite[Prop~2.5]{Zag08}, which seems to require
smoothness of the boundary. A complete proof is announced for a forthcoming
(but not yet accessible) joint paper with Hassan Emamirad.
\end{rem}

In the self-adjoint case we can also prove the following convergence result.
\begin{theo}
Let $\Omega\subset{\mathbb R}^d$ be an open bounded domain with Lipschitz
boundary, where $d \ge 2$. Let $(\alpha_n)_{n\in \mathbb N}\subset
L^\infty(\Omega;\mathbb C^{d \times d})$ be such that $\alpha_n(x)$ is uniformly positive
definite uniformly with respect to $n$, i.e.,
$$(\alpha_n(x)\xi|\xi)\ge k_0 |\xi|^2\qquad \hbox{for a.e. }x\in \Omega,\; \hbox{all }n\in \mathbb N\hbox{ and all }\xi\in\mathbb C^d$$
for some $k_0>0$. Let finally $(\gamma_n)_{n\in\mathbb N}\subset L^q(\Omega;{\mathbb C})$,
$q > \frac{d}{2}$, be such that $\Real \gamma_n \ge -\gamma_0$ a.e.\ for some 
$$\gamma_0 < \inf_{n \in \mathbb{N}} \inf_{u \in H^1_0(\Omega)} \frac{\int_\Omega \alpha_n \nabla u \cdot \overline{\nabla u}}{\|u\|_{L^2(\Omega)}^2}.$$
If $\lim_{n\to \infty}\gamma_n = \gamma$ and $\lim_{n\to \infty}\alpha_n= \alpha$ almost everywhere, then 
$$\lim_{n\to \infty}e^{-tD^{\gamma_n}_{\alpha_n}} = e^{-tD^\gamma_\alpha} \qquad \hbox{in }\mathcal{L}_1(L^2(\partial\Omega))$$
for every $t > 0$.
\end{theo}

\begin{proof}
Define
\[
	b(u,v) \coloneqq k_0 \int_\Omega  \nabla u\cdot \overline{\nabla v} -  \gamma_0 \int_\Omega u \overline{v}
\]
for $u, v \in H^1(\Omega)$, and let $j$ be the trace operator from $H^1(\Omega)$ to $L^2(\partial\Omega)$.
Then $(b,j) \le (a_n,j)$ in the sense of Proposition~\ref{prop:dom} for all $n \in \mathbb{N}$,
where $a_n$ denotes the form associated with $-D^{\gamma_n}_{\alpha_n}$.
Moreover, the semigroup associated with $(b,j)$ is a Gibbs semigroup by Remark~\ref{rem:dtnSM}.
So in view of Theorem~\ref{thm:semiconv} it only remains to show that $D^{\gamma_n}_{\alpha_n} \to D^\gamma_\alpha$
in the strong resolvent sense, for which we employ Remark~\ref{rem:moscosym}.

Let $(u_n)$ be a sequence in $H^1(\Omega)$ such that $u_n|_{\partial\Omega} \rightharpoonup \varphi$ in $L^2(\Omega)$
and $s \coloneqq \liminf a_n(u_n,u_n) < \infty$. Since the constants in the $j_n$-ellipticity of $a_n$
are uniform with respect to $n$, the sequence $(u_n)$ is bounded in $V = H^1(\Omega)$, and thus
we may assume that $u_n \rightharpoonup u$ in $H^1(\Omega)$ for some $u \in H^1(\Omega)$. Then by compactness
$\lim_{n\to \infty}u_n = u$ in $L^2(\Omega)$ and moreover $u_n|_{\partial\Omega} \to u|_{\partial\Omega}$ in $L^2(\partial\Omega)$.
From this and weak lower semicontinuity of the norm in $H^1(\Omega)$ it follows immediately that $a(u,u) \le s$, where $a$ denotes the form associated
with $-D^\gamma_\alpha$.
Moreover, if $u \in H^1(\Omega)$, then clearly $a_n(u_n,u_n) \to a(u,u)$.
Now the convergence follows from Remark~\ref{rem:moscosym}.
\end{proof}

\subsection{Multiplicative perturbations of Laplacians}

Let $\Omega \subset \mathbb{R}^d$ be an open, bounded set.
Let $V = H^1_0(\Omega)$ and $H = L^2(\Omega)$. Define $a(u,v) \coloneqq \int_\Omega \nabla u \overline{\nabla v}$
and $j(u) \coloneqq \frac{u}{m}$, where the real-valued function $m$ on $\Omega$ satisfies $0 < \eps \le m \le M < \infty$  for some constants $\epsilon$ and $M$. Then for the operator $-A_m$ associated
with $(a,j)$ we have $u \in D(A_m)$ with $-A_mu = f$ if and only if $u \in H^1_0(\Omega)$ and $\Delta u = \frac{f}{m}$
distributionally, i.e., at least symbolically, $A_m = -m\Delta$ with Dirichlet boundary conditions.

\begin{theo}
	Let $(m_n)_{n\in \mathbb N}$ be a sequence of measurable functions from $\Omega$ to $\mathbb R$ such that $0 < \eps \le m_n \le M < \infty$
	for all $n \in \mathbb{N}$. If this sequence converges a.e.\ to a measurable function $m:\Omega \to \mathbb R$, then 
	$$\lim_{n\to \infty}e^{tm_n\Delta} = e^{tm\Delta}\qquad \hbox{in }\mathcal{L}_1(L^2(\Omega))$$ 
	for every $t > 0$.
\end{theo}
\begin{proof}
	Comparing with the Gibbs semigroup  generated by $-\eps \Delta$, we see as in the previous section that it suffices to prove
	convergence of $-m_n \Delta$ to $-m\Delta$ in the strong resolvent sense. So take
	a sequence $(u_n)$ in $H^1(\Omega)$ such that $m_n u_n \rightharpoonup m u$ in $L^2(\Omega)$
	and $s \coloneqq \liminf_{n \to \infty} a(u_n,u_n) < \infty$. Then $u_n \rightharpoonup u$
	in $H^1(\Omega)$ after passing to a subsequence, and hence $\lim_{n\to \infty}u_n = u$ by compact embedding,
	which shows in particular that $u \in H^1(\Omega)$. The relation $a(u,u) \le s$ is obvious from weak lower semicontinuity
	of the norm in $H^1(\Omega)$.
	On the contrary, if $u \in H^1(\Omega)$, then $m_n u \to m u$ in $L^2(\Omega)$.
	Hence we obtain convergence in the strong resolvent sense from Theorem~\ref{thm:strongconv}.
\end{proof}

\begin{rem}
For every $k \in \mathbb{N}$, the $k$\textsuperscript{th} eigenvalue $\lambda_k(A_m)$ of $A_m$ is
an increasing function of $m$. More precisely, if $m_1 \le m_2$ almost everywhere, then $\|\frac{u}{m_1}\|_2 \ge \|\frac{u}{m_2}\|_2$
for all $u \in H^1_0(\Omega)$ and hence $\lambda_k(A_{m_1}) \le \lambda_k(A_{m_2})$ by Theorem~\ref{thm:compspec}.
By the way, the operators can in general not be compared in the sense of positive definiteness, as they are not self-adjoint on the same reference space,
so the expression $A_{m_1} \le A_{m_2}$ is not defined and we have to resort to the eigenvalues if we wish to compare the operators in some way.
\end{rem}

\subsection{Comparison of self-adjoint elliptic operators}

Let $\Omega\subset {\mathbb R^d}$ be a bounded open domain with Lipschitz boundary. Let $V:=H^1(\Omega)$ and define
\[
	a(u,v):=\int_\Omega \nabla u \cdot \overline{\nabla v}+\int_{\partial \Omega} \beta u \overline{v} d\sigma
\]
for a given real-valued function $\beta \in L^\infty(\partial \Omega)$. We consider the operators
$j_1 \colon V \to L^2(\Omega)$, $j_2 \colon V \to L^2(\Omega) \times L^2(\partial\Omega)$ and $j_3\colon V \to L^2(\partial\Omega)$
given by $j_1(u) \coloneqq u$, $j_2(u) \coloneqq (u, u_{|\partial\Omega})$ and $j_3(u) \coloneqq u_{|\partial\Omega}$, respectively.
Then $a$ is a $j_k$-elliptic form and
we denote the operator associated with $(a,j_k)$ by $A_k$, $k=1,2,3$. These operators are given by
\begin{align*}
	u \in D(A_1), \; A_1u = f
		& \quad \Leftrightarrow \quad
	\left\{ \begin{aligned}
		-\Delta u & = f \\
		\frac{\partial u}{\partial \nu} + \beta u & = 0
	\end{aligned} \right. \\
	(u,u_{|\partial\Omega}) \in D(A_2), \; A_2(u,u_{|\partial\Omega}) = (f,g)
		& \quad \Leftrightarrow \quad
	\left\{ \begin{aligned}
		-\Delta u & = f \\
		\frac{\partial u}{\partial \nu} + \beta u & = g
	\end{aligned} \right. \\
	\varphi \in D(A_3), \; A_3\varphi = g
		& \quad \Leftrightarrow \quad
	\exists u \in H^1(\Omega) : \left\{ \begin{aligned}
		-\Delta u & = 0 \\
		u_{|\partial\Omega} & = \varphi \\
		\frac{\partial u}{\partial \nu} + \beta u & = g
	\end{aligned} \right.
\end{align*}
where the Laplace operator and the normal derivative are understood in a weak sense,
see~\cite[\S 4.4]{AreEls09} for $A_3$.

Now Theorem~\ref{thm:compspec} yields that
\[
	\lambda_k(A_2) \le \lambda_k(A_1) \quad\text{and}\quad \lambda_k(A_2) \le \lambda_k(A_3)\qquad \hbox{for all }k \in \mathbb{N}. 
\]
These results have also been obtained in~\cite[Thm.~4.2 and Thm.~4.3]{BelFra05} by the same argument.

\subsection{Convergence and non-convergence of Wentzell--Robin operators}\label{sec:wro}
Let $\Omega \subset \mathbb{R}^d$ be a bounded Lipschitz domain and define
$V \coloneqq \{ (u,u|_{\partial\Omega}) : u \in H^1(\Omega) \}$ and $H \coloneqq L^2(\Omega) \times L^2(\partial\Omega)$.
Then
\[
	a( (u,u|_{\partial\Omega}), (v,v|_{\partial\Omega}) ) := \int_\Omega \nabla u \cdot \overline{\nabla v}
\]
defines a bounded sesquilinear form on $V$. We consider the embeddings
\[
	j_{\rho,\sigma}( (u,u|_{\partial\Omega}) ) \coloneqq (\rho u, \sigma u|_{\partial\Omega})
\]
of $V$ into $H$, where $\sigma > 0$ and $\rho > 0$ are constants. Then clearly $a$ is a positive $j_{\rho,\sigma}$-elliptic form,
and the associated operator $A_{\rho,\sigma}$ is (at least on a formal level) given by 
\[
	A_{\rho,\sigma} (u, u|_{\partial\Omega}) = \Bigl( -\frac{1}{\rho} \Delta u, \frac{1}{\sigma} \frac{\partial u}{\partial \nu} \Bigr).
\]

\begin{theo}\label{theo:moscoappl}
The operator $A_{1,\sigma}$ converges to $-\Delta_D \oplus 0$  in the strong resolvent sense as $\sigma \to \infty$,
where $\Delta_D$ denotes the Dirichlet Laplacian on $L^2(\Omega)$.
\end{theo}

\begin{proof}
		The operator $-\Delta_D \oplus 0$ is associated with the $j_D$-elliptic form $a_D$ given by
		$a_D( (u,g), (v,h) ) := \int_\Omega \nabla u  \overline{\nabla v}$,
		where $j_D\colon H^1_0(\Omega) \times L^2(\partial\Omega) \to H$ is given by $j_D( (u,g) ) := (u, g)$.

		Let $\sigma_n \to \infty$ and
		let $(u_n)$ be a sequence in $V$ such that
		\[
			j_n(u_n) \coloneqq (u_n, \sigma_n u_n|_{\partial\Omega}) \rightharpoonup (u,g)
		\]
		in $H$ and $s \coloneqq \liminf \int_\Omega |\nabla u_n|^2 < \infty$.
		Passing to a subsequence we can assume that $\int_\Omega |\nabla u_n|^2 \to s$. Then $(u_n)$ is bounded
		in $H^1(\Omega)$, and passing to further subsequence we can assume that $u_n \rightharpoonup u$ in $H^1(\Omega)$.
		Then in particular $u_n|_{\partial\Omega} \to u|_{\partial\Omega}$ and $\int_\Omega |\nabla u|^2 \le s$.
		Moreover,
		\[
			u_n|_{\partial\Omega} = \frac{\sigma_n u_n|_{\partial\Omega}}{\sigma_n} \to 0
		\]
		since $(\sigma_n u_n|_{\partial\Omega})$ is bounded and $\sigma_n \to \infty$, hence $u|_{\partial\Omega} = 0$.
		Thus $(u,g) \in H^1_0(\Omega) \times L^2(\partial\Omega)$, $j_D( (u,g) ) = (u,g)$ and
		$\liminf \int_\Omega |\nabla u_n|^2 \ge \int_\Omega |\nabla u|^2$.
		We have checked the first part of the characterisation in Theorem~\ref{thm:strongconv}.

		For the second part, let $(u,g) \in H^1_0(\Omega) \times L^2(\partial\Omega)$ be fixed.
		Since $\sigma_n \to \infty$, there exist $v_n \in H^1(\Omega)$ satisfying $v_n|_{\partial\Omega} \to g$ in $L^2(\partial\Omega)$
		and $\frac{v_n}{\sigma_n} \to 0$ in $H^1(\Omega)$. Define $u_n \coloneqq u + \frac{v_n}{\sigma_n}$. Then
		$(u_n, u_n|_{\partial\Omega}) \in V$ and
		\[
			j_n( (u_n, u_n|_{\partial\Omega}) ) = (u_n, v_n|_{\partial\Omega}) \to (u, g) = f_D( (u,g) )
		\]
		in $H$. Moreover, $\int_\Omega |\nabla u_n|^2 \to \int_\Omega |\nabla u|^2$ since $u_n \to u$ in $H^1(\Omega)$.
\end{proof}

As already emphasised, one advantage of our Mosco-type result is that it \emph{characterises} convergence, meaning that it paves the road to \emph{non-convergence} results as well.
Given that the eigenvalue problem associated with the operator $A_{\rho,\sigma}$ is
\[
\left\{ \begin{aligned}
		\lambda\rho u&=\Delta u\qquad &\hbox{in }\Omega,\\
		\lambda \sigma u_{|\partial\Omega} & =-\frac{\partial u}{\partial \nu}\qquad  &\hbox{on }\partial \Omega,
	\end{aligned} \right.
\]
while the eigenvalue problem associated with the Dirichlet-to-Neumann operator is
\[
\left\{ \begin{aligned}
		0&=\Delta u\qquad &\hbox{in }\Omega,\\
		\lambda u_{|\partial\Omega} & =-\frac{\partial u}{\partial \nu}\qquad  &\hbox{on }\partial \Omega,
	\end{aligned} \right.
\]
the following may look surprising.

\begin{prop}
The following assertions hold in the space $L^2(\Omega)\times L^2(\partial \Omega)$.
	\begin{enumerate}[(1)]
	\item	
		$A_{1,\sigma}$ does not converge to any closed operator in the weak resolvent sense as $\sigma \to 0$.
	\item
		$A_{\rho,1}$ does not converge to any closed operator in the weak resolvent sense as $\rho \to 0$.
	\end{enumerate}
\end{prop}

\begin{proof}
In both cases, we follow the same strategy. Assume that the family of operators converges
to a densely defined (necessarily self-adjoint) operator $B$ on $H:=L^2(\Omega)\times L^2(\partial \Omega)$ in the weak resolvent sense.
Then the operators converge even in the strong resolvent sense~\cite[\S VIII.7]{ReeSim80},
and hence the quadratic forms converge in the sense of Mosco by Theorem~\ref{thm:strongconv}.
But for both situations we will show that the set of $u$ such that the
second condition of part (b) of Theorem~\ref{thm:strongconv} can be satisfied is non-dense in $H$.
Hence there cannot be a limiting quadratic form, thus proving the claim.

(1)
		Take a null sequence $(\sigma_n)_{n\in\mathbb N}$. Let $(u_n)$ be a sequence in $H^1(\Omega)$
		such that $j_n( (u_n, u_n|_{\partial\Omega}) ) = (u_n, \sigma_n u_n|_{\partial\Omega})$
		converges to $(u,g)$ in $H$. Assume moreover that $a( (u_n,u_n|_{\partial\Omega}), (u_n, u_n|_{\partial\Omega}) )$
		is bounded. Then $(u_n)_{n\in\mathbb N}$ is bounded in $H^1(\Omega)$.
		Hence $(u_n|_{\partial\Omega})_{n\in\mathbb N}$ is bounded in $L^2(\partial\Omega)$, implying that $\sigma_n u_n|_{\partial\Omega} \to 0$,
		i.e., $g = 0$. Hence the set of possible limits in (b.ii) of Theorem~\ref{thm:strongconv} is
		contained in the non-dense set $L^2(\Omega) \times \{0\}$.

(2)
		Let $(\rho_n)_{n\in\mathbb N}$ be a null sequence. Let $(u_n)$ be a sequence in $H^1(\Omega)$
		such that $j_n( (u_n, u_n|_{\partial\Omega}) ) = (\rho_n u_n, u_n|_{\partial\Omega})$
		converges to $(u,g)$ in $H$. Assume moreover that $a( (u_n,u_n|_{\partial\Omega}), (u_n, u_n|_{\partial\Omega}) )$
		is bounded. Then $(u_n)_{n\in\mathbb N}$ is bounded in $H^1(\Omega)$ since
		for some $c > 0$ we have
		\[
			\|u\|_{H^1(\Omega)}^2 \le c \|\nabla u\|_{L^2(\Omega; \mathbb{R}^N)}^2 + c \|u|_{\partial\Omega}\|_{L^2(\partial\Omega)}^2
		\]
		by~\cite[\S 1.1.15]{Maz85}.
		Hence $(u_n)_{n\in\mathbb N}$ is bounded in $L^2(\Omega)$, implying that $\rho_n u_n \to 0$,
		i.e., $u = 0$. Hence the set of possible limits in (b.ii) of Theorem~\ref{thm:strongconv} is
		contained in the non-dense set $\{0\} \times L^2(\partial\Omega)$.
	\end{proof}

\bibliographystyle{abbrvnat}
\bibliography{../referenzen/literatur}

\begin{thebibliography}{45}
\providecommand{\natexlab}[1]{#1}
\providecommand{\url}[1]{\texttt{#1}}
\expandafter\ifx\csname urlstyle\endcsname\relax
  \providecommand{\doi}[1]{doi: #1}\else
  \providecommand{\doi}{doi: \begingroup \urlstyle{rm}\Url}\fi

\bibitem[Adams(1975)]{Ada75}
R.~Adams.
\newblock \emph{{Sobolev Spaces}}.
\newblock Pure Appl. Math. Academic Press, New York, 1975.

\bibitem[Arendt(2006)]{Are06}
W.~Arendt.
\newblock Heat {K}ernels -- {M}anuscript of the $9^{\rm th}$ {I}nternet
  {S}eminar, 2006.
\newblock Freely available at
  {\burl{http://www.uni-ulm.de/fileadmin/website_uni_ulm/mawi.inst.020/arendt/%
downloads/internetseminar.pdf}}.

\bibitem[Arendt and ter Elst(2009)]{AreEls09}
W.~Arendt and T.~ter Elst.
\newblock Sectorial forms and degenerate differential operators.
\newblock arXiv:0812.3944, 2009.

\bibitem[Arendt and ter Elst(2010)]{AreEls10}
W.~Arendt and T.~ter Elst.
\newblock The {D}irichlet-to-{N}eumann operator on rough domains.
\newblock arXiv:1005.0875v1, 2010.

\bibitem[Arendt and ter Elst(2011)]{AreEls11}
W.~Arendt and T.~ter Elst.
\newblock From forms to semigroups.
\newblock arXiv:1104.1013, 2011.

\bibitem[Arendt et~al.(2001)Arendt, Batty, Hieber, and Neubrander]{AreBatHie01}
W.~Arendt, C.~Batty, M.~Hieber, and F.~Neubrander.
\newblock \emph{Vector-{V}alued {L}aplace {T}ransforms and {C}auchy
  {P}roblems}, volume~96 of \emph{Monographs in Mathematics}.
\newblock Birkh{\"a}user, Basel, 2001.

\bibitem[Arendt et~al.(2003)Arendt, Metafune, Pallara, and
  Romanelli]{AreMetPal03}
W.~Arendt, G.~Metafune, D.~Pallara, and S.~Romanelli.
\newblock The {L}aplacian with {W}entzell--{R}obin boundary conditions on
  spaces of continuous functions.
\newblock \emph{Semigroup Forum}, 67:\penalty0 247--261, 2003.

\bibitem[Attouch(1984)]{Att84}
H.~Attouch.
\newblock \emph{Variational convergence for functions and operators}.
\newblock Applicable Mathematics Series. Pitman, Boston, 1984.

\bibitem[Aubin(1982)]{Aub82}
T.~Aubin.
\newblock \emph{{Nonlinear analysis on manifolds, Monge-Ampere equations}},
  volume 252 of \emph{Grundlehren der mathematischen Wissenschaften}.
\newblock Springer-Verlag, Berlin, 1982.

\bibitem[Behrndt et~al.(2010)Behrndt, Langer, Lobanov, Lotoreichik, and
  Popov]{BehLanLob10}
J.~Behrndt, M.~Langer, I.~Lobanov, V.~Lotoreichik, and I.~Popov.
\newblock {A remark on Schatten-von Neumann properties of resolvent differences
  of generalized Robin Laplacians on bounded domains}.
\newblock \emph{J. Math. Anal. Appl.}, 371:\penalty0 750--758, 2010.

\bibitem[Below and Fran{\c c}ois(2005)]{BelFra05}
J.~Below and G.~Fran{\c c}ois.
\newblock Spectral asymptotics for the {L}aplacian under an eigenvalue
  dependent boundary condition.
\newblock \emph{Bull. Belg. Math. Soc. - Simon Stevin}, 12:\penalty0 505--519,
  2005.

\bibitem[Birman(2008)]{Bir08}
M.~Birman.
\newblock {Perturbations of the continuous spectrum of a singular elliptic
  operator by varying the boundary and the boundary conditions}.
\newblock In T.~Suslina and D.~Yafaev, editors, \emph{Spectral theory of
  differential operators: M. Sh. Birman 80th anniversary collection}, volume
  225 of \emph{Adv. Math. Sciences}, pages 19--54. Amer. Math. Soc., 2008.

\bibitem[Cardanobile and Mugnolo(2009)]{CarMug09}
S.~Cardanobile and D.~Mugnolo.
\newblock Parabolic systems with coupled boundary conditions.
\newblock \emph{J. Differ. Equ.}, 247:\penalty0 1229--1248, 2009.

\bibitem[Cheon et~al.(2010)Cheon, Exner, and Turek]{CheExnTur10}
T.~Cheon, P.~Exner, and O.~Turek.
\newblock {Approximation of a general singular vertex coupling in quantum
  graphs}.
\newblock \emph{Ann. Phys.}, 325:\penalty0 548--578, 2010.

\bibitem[Clark(1966)]{Cla66}
C.~Clark.
\newblock {The Hilbert--Schmidt property for embedding maps between Sobolev
  spaces}.
\newblock \emph{Can. J. Math.}, 18:\penalty0 1079--1084, 1966.

\bibitem[Coclite et~al.(2008)Coclite, Goldstein, and Goldstein]{CocGolGol08}
G.~Coclite, G.~Goldstein, and J.~Goldstein.
\newblock {Stability estimates for parabolic problems with Wentzell boundary
  conditions}.
\newblock \emph{J. Differ. Equ.}, 245\penalty0 (9):\penalty0 2595--2626, 2008.

\bibitem[Crouzeix(2004)]{Cro04}
M.~Crouzeix.
\newblock {Operators with numerical range in a parabola}.
\newblock \emph{Arch. Math.}, 82:\penalty0 517--527, 2004.

\bibitem[Demuth et~al.(1995)Demuth, Stollmann, Stolz, and van
  Casteren]{DemStoSto95}
M.~Demuth, P.~Stollmann, G.~Stolz, and J.~van Casteren.
\newblock {Trace norm estimates for products of integral operators and
  diffusion semigroups}.
\newblock \emph{Integral Equations Oper. Theory}, 23:\penalty0 145--153, 1995.

\bibitem[Desch and Schappacher(1984)]{DesSch84}
W.~Desch and W.~Schappacher.
\newblock On relatively bounded perturbations of linear ${C}_0$-semigroups.
\newblock \emph{Ann. Sc. Norm. Super. Pisa, Cl. Sci.}, 11:\penalty0 327--341,
  1984.

\bibitem[Desch and Schappacher(1988)]{DesSch88}
W.~Desch and W.~Schappacher.
\newblock Some perturbation results for analytic semigroups.
\newblock \emph{Math. Ann.}, 281:\penalty0 157--162, 1988.

\bibitem[Favini et~al.(2003)Favini, Goldstein, Goldstein, Obrecht, and
  Romanelli]{FavGolGol03}
A.~Favini, G.~Goldstein, J.~Goldstein, E.~Obrecht, and S.~Romanelli.
\newblock {The Laplacian with generalized Wentzell boundary conditions}.
\newblock In M.~Iannelli and G.~Lumer, editors, \emph{Evolution Equations 2000:
  Applications to Physics, Industry, Life Sciences and Economics (Proceedings
  Levico Terme 2000)}, volume~55 of \emph{Progress in Nonlinear Differential
  Equations}, pages 169--180, Basel, 2003. Birk{\"a}user.

\bibitem[Gapaillard(1974)]{Gap74}
C.~Gapaillard.
\newblock Un resultat de compacit\'e pour l'interpolation de couples
  hilbertiens.
\newblock \emph{C.R. Acad. Sc. Paris S\'er. A}, 278:\penalty0 681--684, 1974.

\bibitem[Gesztesy and Mitrea(2009)]{GesMit09}
F.~Gesztesy and M.~Mitrea.
\newblock A description of all self-adjoint extensions of the {L}aplacian and
  {K}rein-type resolvent formulas in nonsmooth domains.
\newblock arXiv:0907.1750, 2009.

\bibitem[Gohberg and Kre{\u\i}n(1969)]{GohKre69}
I.~Gohberg and M.~Kre{\u\i}n.
\newblock \emph{Introduction to the theory of linear nonselfadjoint operators},
  volume~18 of \emph{Transl. Math. Monographs}.
\newblock Amer. Math. Soc., Providence, RI, 1969.

\bibitem[Gramsch(1968)]{Gra68}
B.~Gramsch.
\newblock {Zum Einbettungssatz von Rellich bei Sobolevr\"aumen}.
\newblock \emph{Math. Z.}, 106:\penalty0 81--87, 1968.

\bibitem[Grisvard(1985)]{Gri85}
P.~Grisvard.
\newblock \emph{{Elliptic Problems in Nonsmooth Domains}}, volume~24 of
  \emph{Monographs and Studies in Mathematics}.
\newblock Pitman, Boston, 1985.

\bibitem[Kato(1995)]{Kat95}
T.~Kato.
\newblock \emph{{Perturbation Theory for Linear Operators}}.
\newblock Classics in Mathematics. Springer-Verlag, New York, 1995.

\bibitem[K{\"o}nig(1977)]{Koe77}
H.~K{\"o}nig.
\newblock Operator properties of {S}obolev imbeddings over unbounded domains.
\newblock \emph{J. Funct. Anal.}, 24:\penalty0 32--51, 1977.

\bibitem[Lions and Magenes(1972)]{LioMag72}
J.~Lions and E.~Magenes.
\newblock \emph{Non-{H}omogeneous {B}oundary {V}alue {P}roblems and
  {A}pplications}, volume 181--183 of \emph{Grundlehren der mathematischen
  Wissenschaften}.
\newblock Springer-Verlag, Berlin, 1972.

\bibitem[Maurin(1961)]{Mau61}
K.~Maurin.
\newblock {Abbildungen vom Hilbert--Schmidtschen Typus und ihre Anwendungen}.
\newblock \emph{Math. Scand.}, 9:\penalty0 359--371, 1961.

\bibitem[Maz'ya(1985)]{Maz85}
V.~Maz'ya.
\newblock \emph{Sobolev {S}paces}.
\newblock Springer-Verlag, Berlin, 1985.

\bibitem[Mosco(1967)]{Mos67}
U.~Mosco.
\newblock {Approximation of the solutions of some variational inequalities}.
\newblock \emph{Ann. Sc. Norm. Super. Pisa, Cl. Sci., III Serie}, 21:\penalty0
  373--394, 1967.

\bibitem[Mugnolo(2008)]{Mug08}
D.~Mugnolo.
\newblock A variational approach to strongly damped wave equations.
\newblock In H.~{Amann et al.}, editor, \emph{Functional Analysis and Evolution
  Equations -- The G{\"u}nter Lumer Volume}, pages 503--514. Birkh{\"a}user,
  Basel, 2008.

\bibitem[Mugnolo et~al.(2010)Mugnolo, Nittka, and Post]{MugNitPos10}
D.~Mugnolo, R.~Nittka, and O.~Post.
\newblock Convergence of sectorial operators on varying {H}ilbert spaces.
\newblock arXiv:1007.3932, 2010.

\bibitem[Neidhardt and Zagrebnov(1990)]{NeiZag90b}
H.~Neidhardt and V.~A. Zagrebnov.
\newblock The {T}rotter-{K}ato product formula for {G}ibbs semigroups.
\newblock \emph{Comm. Math. Phys.}, 131\penalty0 (2):\penalty0 333--346, 1990.

\bibitem[Ouhabaz(2005)]{Ouh05}
E.~Ouhabaz.
\newblock \emph{Analysis of {H}eat {E}quations on {D}omains}, volume~30 of
  \emph{Lond. Math. Soc. Monograph Series}.
\newblock Princeton Univ. Press, Princeton, 2005.

\bibitem[Post(2006)]{Pos06}
O.~Post.
\newblock Spectral convergence of quasi-one-dimensional spaces.
\newblock \emph{Ann. Henri Poincar\'e}, 7:\penalty0 933--973, 2006.

\bibitem[Reed and Simon(1980)]{ReeSim80}
M.~Reed and B.~Simon.
\newblock \emph{Methods of modern mathematical physics. {I}}.
\newblock Academic Press Inc. [Harcourt Brace Jovanovich Publishers], New York,
  second edition, 1980.
\newblock ISBN 0-12-585050-6.
\newblock Functional analysis.

\bibitem[Schatten and von Neumann(1946)]{SchNeu46}
R.~Schatten and J.~von Neumann.
\newblock {The cross-space of linear transformations. II}.
\newblock \emph{Ann. Math.}, 47:\penalty0 608--630, 1946.

\bibitem[Simon(2005)]{Sim05}
B.~Simon.
\newblock \emph{{Trace ideals and their applications}}, volume 120 of
  \emph{Math. Surveys and Monographs}.
\newblock Amer. Math. Soc., Providence, RI, 2005.

\bibitem[Stollmann(1994)]{Sto94}
P.~Stollmann.
\newblock {Trace ideal properties of perturbed Dirichlet semigroups}.
\newblock In \emph{Mathematical results in quantum mechanics (Proc. Blossin
  1993)}, volume~70 of \emph{Oper. Theory, Adv. Appl.}, pages 153--158, Basel,
  1994. Birkh\"auser.

\bibitem[Travis and Webb(1977)]{TraWeb77}
C.~Travis and G.~Webb.
\newblock {Compactness, regularity, and uniform continuity properties of
  strongly continuous cosine families}.
\newblock \emph{Houston J. Math}, 3:\penalty0 555--567, 1977.

\bibitem[Uhlenbrock(1971)]{Uhl71}
D.~Uhlenbrock.
\newblock {Perturbation of statistical semigroups in quantum statistical
  mechanics}.
\newblock \emph{J. Math. Phys.}, 12:\penalty0 2503, 1971.

\bibitem[Zagrebnov(1980)]{Zag80}
V.~Zagrebnov.
\newblock {On the families of Gibbs semigroups}.
\newblock \emph{Commun. Math. Phys.}, 76:\penalty0 269--276, 1980.

\bibitem[Zagrebnov(2008)]{Zag08}
V.~A. Zagrebnov.
\newblock {{From Laplacian transport to Dirichlet-to-Neumann (Gibbs)
  semigroups}}.
\newblock \emph{Zh. Mat. Fiz. Anal. Geom.}, 4:\penalty0 551--568, 2008.

\end{thebibliography}
\end{document}